%% file: SLLN_torus_v4.tex
\documentclass[reqno, 11pt]{amsart}
\usepackage{url}
\usepackage[colorlinks, linkcolor=blue, citecolor=blue, urlcolor=blue]{hyperref}
\usepackage[usenames,dvipsnames]{xcolor}

\usepackage{verbatim}
\usepackage{float}
\usepackage{framed}
\usepackage{mdframed}
\usepackage{stmaryrd}

\binoppenalty=\maxdimen
\relpenalty=\maxdimen

\usepackage{times}

\usepackage{amsmath}
\usepackage{amsthm}
\usepackage{amssymb}
\usepackage{mathrsfs}
\usepackage{amsxtra}
\usepackage{pifont}
\usepackage{amsbsy}

\usepackage{pb-diagram}
\usepackage{lamsarrow}
\usepackage{pb-lams}

\usepackage{tikz}
\makeatletter
\newcommand\mathcircled[1]{%
  \mathpalette\@mathcircled{#1}%
}
\newcommand\@mathcircled[2]{%
  \tikz[baseline=(math.base)] \node[draw,circle,inner sep=1pt] (math) {$\m@th#1#2$};%
}
\makeatother

\numberwithin{equation}{section}
\numberwithin{table}{section}
\numberwithin{figure}{section}

\theoremstyle{plain}
\newtheorem{theorem}{Theorem}[section]
\newtheorem{lemma}[theorem]{Lemma}

\newtheorem{proposition}[theorem]{Proposition}

\newtheorem{corollary}[theorem]{Corollary}

\theoremstyle{definition}
\newtheorem{definition}[theorem]{Definition}

\newtheorem{ex}[theorem]{Example}
\newtheorem{remark}[theorem]{Remark}

\numberwithin{theorem}{section}

\usepackage{bbm}
\usepackage{euscript}

\newcommand{\ga}{\mathfrak{a}}

\input{Prob_definitions}

\begin{document}

\title[A law of large numbers]{A    law of large numbers concerning the distribution of critical points of random Fourier series} 

\author{Qiangang ``Brandon'' Fu}
\address{Brandon Fu: Department of Mathematics, University of Notre Dame, Notre Dame, IN 46556-4618.}
\email{qfu3@nd.edu}

\author{Liviu I. Nicolaescu}
\address{Liviu  Nicolaescu: Department of Mathematics, University of Notre Dame, Notre Dame, IN 46556-4618.}
\email{lnicolae@nd.edu}
\urladdr{\url{http://www.nd.edu/~lnicolae/}}
\thanks{Last revised \today}

\begin{abstract} On the flat torus $\mathbb{T}^m=\mathbb{R}^m/\mathbb{Z}^m$  with angular coordinates $\vec{\theta}$ we consider  the  random function $F_R=\mathfrak{a}\big(\, R^{-1} \sqrt{\Delta}\,\big) W$, where  $R>0$,  $\Delta$ is the Laplacian on this flat torus,  $\mathfrak{a}$ is an even Schwartz function on $\mathbb{R}$ such that $\mathfrak{a}(0)>0$ and $W$ is the Gaussian white noise on $\mathbb{T}^m$ viewed as a random generalized function. For any  $f\in C(\mathbb{\bT}^m)$ we set
\[
Z_R(f):=\sum_{\nabla F_R(\vec{\theta})=0} f(\vec{\theta})
\]
We prove that if the support of $f$ is contained in a geodesic ball of $\mathbb{T}^m$, then  the variance of $Z_R(f)$ is asymptotic to  $const\times R^{m}$ as $R\to\infty$. We use  this to prove that  if $m\geq 2$,  then as $N\to\infty$   the random measures $N^{-m}Z_N(-)$ converge  a.s. to  an explicit multiple of the volume measure on the flat torus. \end{abstract}

\maketitle

\tableofcontents

\section{Introduction}

Denote by $\bT^m$ the $m$-dimensional  torus $\bR^m/\bZ^m$  and by $g_1$ the flat metric of volume $1$. In terms of angular coordinates $\vtheta=(\theta^1,\dotsc, \vtheta^m)$, $\theta^i\in \bR\bmod\bZ$,
\[
g_1= (d\theta^1)^2+\cdots +(d\theta^m)^2.
\]
For $R>0$, meant to be large,  we  denote by $\Delta_R$ the Laplacian of the metric $g_R=R^2g_1$. Observe that 
\[
\vol\lb M,g_R\rb=R^m\vol\lb M,g_1\rb=R^m,\;\;\Delta_R=R^{-2}\Delta_1.
 \]
A complete orthonormal system of \emph{complex} eigenfunctions of $\Delta_1$ is given by $\lp e_\vell\rp_{\vell\in \bZ^m}$
\[
e_\vell\lp \vtheta\rp= e^{2\pi\ii \lan\vell,\vtheta\ran},\;\; \lan\vell,\vtheta\ran=\sum_{j=1}^m \ell_j\theta^j.
\]
To describe a complete orthonormal system of \emph{real} eigenfunctions of $\Delta_1$ we introduce the lexicografic order on $\bZ^m$, $\vell\succ 0$ iff  $\exists i_0:\;\;\ell_{i_0}>0,\;\;\ell_i=0,\;\;\forall i<i_0$. We set $u_0=1$ and, for $\vell\succ 0$ we  set
\[
u_\vk\lp \vtheta\rp=\sqrt{2} \cos\lp 2\pi\lan \vk,\vtheta\ran\rp,\;\; v_\vell\lp \vtheta\rp=\sqrt{2} \sin\lp 2\pi\lan \vell,\vtheta\ran\rp,
\]
\[
u^R_\vk=R^{-m/2}u_\vk,\;\;v_\vell^R=R^{-m/2}v_\vell.
\]
The collection 
\[
\lbr u^R_\vk,\;v^R_\vell;\;\vk\succeq 0,\;\;\vell\succ 0\rbr
\]
is a complete $L^2(M,g_R)$-orthonormal system  of real eigenfunctions of $\Delta_R$.    Moreover
\[
\Delta_R u^R_\vk=\lambda_\vk(R),\;\;\Delta_R v^R_\vell=\lambda_\vell(R)v^R_\vell,\;\;\lambda_\vk(R)=R^{-2}\lv 2\pi \vk\rv^2=\lp 2\pi/R\rp^2\sum_{j=1}^mk_j^2.
\]
Fix  an even  Schwarz function $\ga\in \eS(\bR)$ such that $\ga(0)=1$. We  will refer to such an $\ga$ as  \emph{amplitude}.   Fix independent  standard normal random variables
\[
\lbr A_\vk,\;B_\vell;\;\vk\succeq 0,\;\;\vell\succ 0\rbr
\]
and  consider the random Fourier series
\begin{equation}\label{rand_fourier0}
\begin{split}
F^R_\ga(\vtheta)=\ga(0) A_0 u^R_0 \lp\vtheta\rp+\sum_{\vell\succ 0} \ga\lp \lambda_\vell(R)^{1/2}\rp \lp A_\vell u_\vell^R\lp \vtheta\rp+B_\vell v^R_\vell\lp \vtheta\rp\rp\\
=R^{-m/2}\Lp A_0 u_0\lp \vtheta\rp+\sum_{\vell\succ 0}\ga\lp |2\pi\vell|/R\rp\lp A_\vell u_\vell\lp \vtheta\rp+B_\vell v\lp \vtheta\rp\rp\rp\\
=R^{-m/2} \sum_{\vell \in\bZ^m}  \ga\lp |2\pi\vell|/R\rp Z_\vell e_\vell(\vtheta),\hspace{4cm}
\end{split}
\end{equation}
where\[
Z_\vell=\begin{cases} 
A_0,& \vell=0,\\
\frac{1}{\sqrt{2}}\lp A_\vell-\ii B_\vell\rp, &\vell\succ 0,\\
\bar{Z}_{-\vell}, &\vell\prec 0.
\end{cases}
\]
Since $ \ga\lp |2\pi\vell|/R\rp$ decays very fast as $|\vell|\to \infty$   we deduce   from Kolmogorov's two-series theorem  that for any $\nu\in \bN$ the random series
\[
 \sum_{\vell \in\bZ^m}   \ga\lp |2\pi\vell|/R\rp^2 \,\Vert\, e_\vell\,\Vert^2_{C^\nu(\bT^m)}
 \]
 converges a.s.  and thus the series
 \[
 \sum_{\vell \in\bZ^m}  \ga\lp |2\pi\vell|/R\rp Z_\vell e_\vell
 \]
 converges  a.s. in $C^\nu\lp \bT^m\rp$. In particular, this shows that the Gaussian function $F^R_\ga$ is  a.s. smooth. Its covariance kernel is
 \[
 \eC_\ga^R\lp \vvfi+\vtau,\vvfi\rp=C^R_\ga(\vtau)=R^{-m}\sum_{\vell\in\bZ^m}  \ga\lp |2\pi\vell|/R\rp^2 e^{2\pi \ii\lan \vell, \vtau\ran}.
 \]
 Define
\[
w_\ga:\bR^m\to\bR,\;\;w_\ga\lp\xi\rp=\ga\lp| \xi| \rp^2,\;\;|\xi|^2:=\sum_{j=1}^m\xi_j^2.
\]
Its Fourier transform is
\[
\widehat{w}_\ga(\bx)=\int_{\bR^m}e^{\ii\lan\xi,\bx\ran} w_\ga(\xi) d\xi.
\]
 The Fourier inversion formula shows that
 \[
 \ga\lp |\xi|^2\rp=\int_{\bR^m}e^{\ii\lan \xi,\bx\ran} K_\ga(\bx)d\bx,\;\;\bsK_\ga(\bx):=\frac{1}{(2\pi)^m}  \widehat{w}_\ga(\bx).
 \]
 Using Poisson's summation formula  \cite[\S 7.2]{H1} we deduce
\begin{equation}\label{heat_ker_torus_w}
 C_\ga^R(\vtau)=\sum_{\vk\in \bZ^m} \bsK_\ga\lp(\vk-\vtau)R\rp,\;\;\vtau=\vtheta-\vvfi.
 \end{equation}
 We can think of $F^R_\ga$ either as a function on $\bT^m$, or  as a $\bZ^m$-periodic  function of $\bR^m$.  If we formally let $R\to\infty$ in the equality 
 \[
 R^{m/2}F_\ga^R(\vtheta)=\sum_{\vell \in\bZ^m}  \ga\lp |2\pi\vell|/R\rp Z_\vell e_\vell(\vtheta)
 \]
 we deduce
 \[
 W_\infty(\vtheta) ``="\lim_{R\to\infty}  R^{m/2}F_\ga^R(\vtheta)= \sum_{\vell \in\bZ^m}  Z_\vell e_\vell(\vtheta).
 \]
The series on the right-hand-side is $\as$ divergent  but we can still assign  a meaning to $W_\infty$  as a  random generalized  function,  i.e., a random linear functional  
\[
C^\infty(\bT^m)\to\bR,\;\;W_\infty(f)=  \sum_{\vell \in\bZ^m}  Z_\vell \lp f, e_\vell(\vtheta)\rp_{L^2(\bT^m)}.
\]
A simple computation shows that  for any  functions $f_0,f_1\in C^\infty(\bT^m)$
\[
\cov\lb W_\infty(f_0),W_\infty(f_1)\rb=\sum_{\vell\in\bZ^m} \lp f_0, e_\vell\rp_{L^2(\bT^m,g_1)}\lp f_1, e_\vell\rp_{L^2(\bT^m,g_1)}=\lp f_0, f_1\rp_{L^2(\bT^m,g_1)}.
\]
 The last equality  shows  that  $W_\infty$   is   the  Gaussian white noise  on $\bT^m$ driven by the volume measure $\vol_{g_1}$; see   \cite{GeVil_4}.  In other words, one could think of  the family $\lp W_R=R^{m/2}F^R_\ga\rp_{R>0}$ as a white noise approximation. Note that $W_\ga^R= \ga\lp R^{-1}\sqrt{\Delta}\rp W_\infty$.   In the special case  $\ga(x)=e^{-x^2}$, $t=R^{-2}$,  we have  $\ga\lp R^{-1}\sqrt{\Delta}\rp =e^{-t\Delta}$, the heat operator.

 The main goal of this paper is to investigate the distribution  on $\bT^m$ of  the critical points of $F_\ga^R$  in the white noise limit,  $R\to\infty$.

With this in mind, we consider  the rescaled function 
 \[
\Phi^R_\ga\lp \bx\rp:= F^R_\ga \lp \bx/R\rp= R^{-m/2} \sum_{\vell \in\bZ^m}  \ga\lp |2\pi\vell|/R\rp Z_\vell {e^{2\pi\ii \lan\vell,R^{-1}\bx\ran}}.
 \]
  We denote by  ${\eK}^R_\ga$  the covariance kernel of  $\Phi_\ga^R$. Then   ${\eK}_\ga^R(\bx,\by)={\bsK}_\ga^R(\bx-\by)$, where
  \begin{equation}\label{check_cov}
 {\bsK}_\ga^R(\bz) \stackrel{(\ref{heat_ker_torus_w})}{=}\sum_{\vk\in \bZ^m} \bsK_\ga (R\vk-\bz)=\sum_{\bt\in (R\bZ)^m} \bsK_\ga (\bt-\bz).
 \end{equation}
Since $\bsK_\ga$ is a Schwartz function we deduce that
 \begin{equation}\label{semi}
 \lim_{R\nearrow \infty}{\bsK}_\ga^R=\bsK_\ga\;\;\mbox{ in $C^k\lp \bR^m\rp$, $\forall k \in\bN$}.
 \end{equation}
The function  $\bsK_\ga(\bx-\by)$ is the covariance kernel of  an isotropic Gaussian function $\Phi_\ga$.  The equality (\ref{semi}) suggests that $\Phi_\ga^R$ approximates $\Phi_\ga$ for $R>>0$.  

Suppose that $G:\bR^m\to\bR$ is a Gaussian $C^2$-function such that $\nabla G(\bx)$ is a nondegenerate Gaussian vector for any $\bx\in \bR^m$. Then $G$ is $\as$ Morse (see Corollary \ref{cor: KR_crit}). Define
\[
\fC\lb-,G\rb=\sum_{\nabla G(\bx)=0}\delta_\bx.
\]
This is a locally finite random measure on $\bR^m$ in the sense of \cite{DaVere2} or \cite{Kalle_RM}.  Thus, for every Borel subset $S\subset \bR^m$,  $\fC[S, G]$  is the number of critical points of $G$  in $S$.  More generally, for any measurable function $\vfi:\bR^m\to[0,\infty)$, we set
\[
\fC[ \vfi,G]:=\int_{\bR^m} \vfi(\bx)\fC[d\bx, G]=\sum_{\nabla G(\bx)=0} \vfi(\bx)\in [0,\infty].
\]
Clearly, $\nabla F_\ga^R(\by)=0$ iff  $\nabla \Phi_\ga^R (R\by)=0$ so, for any box $B=[a,b]^m\subset \bR^m$, and any $f\in C^0_\cpt(\bR^m)$, we have
  \[
  \fC\lb B, \Phi_\ga^R\rb=\fC\lb RB, F_\ga^R\rb,\;\; \fC\lb f, F_\ga^R\rb=\fC\lb f_R, \Phi_\ga^R\rb
  \]
  where $f_R(\bx)=f\lp R^{-1}\bx\rp$.
  
  When  the support of $f$ is contained in a fundamental domain of the $\bZ^m$-action on $\bR^m$  we can identify $f$ canonically with a function $\bar{f}$ on $\bT^m$ and we have
  \[
  \fC\lb f, F_\ga^R\rb=\sum_{\substack{\bx\in \bT^m,\\ \nabla F_\ga^R(\bx)=0}}\bar{f}(\bx)= \sum_{\substack{\bx\in \bR^m,\\ \nabla \Phi_\ga^R(\bx)=0}}f(\bx).
  \]

In \cite{Nico_var, Nico_SLLN} it is shown that the  random function $\Phi_\ga$ is $\as$ Morse and there exists an explicit positive constant $C_m(\ga)$ that depends only on $\ga$ and $m$ such that for any box  $B \subset \bR^m$ and any $f\in C^0_\cpt(\bR^m)$ we have
\begin{subequations}
\begin{equation}\label{cr_counta}
\bE\lb\fC[B,\Phi_\ga]\rb= C_m(\ga)\vol\lb B\rb,
\end{equation}
\begin{equation}\label{cr_countb}
\bE\lb \fC[f,\Phi_\ga]\rb =C_m(\ga)\int_{\bR^m} f(\bx)\blam\lb dx\rb,
\end{equation}
\end{subequations}
where $\blam$ denotes the Lebesgue measure on $\bR^m$.

 Set  $C_1:=[0,1]^m$. In  \cite{Nico_var} the second author proved that  there exists a constant $C_\ga'(m)\geq 0$ such that
 \begin{equation}\label{var_asymp}
 \lim_{R\to\infty}R^{-m} \var\lb \fC[C_1, F_\ga^R]\rb=C'_m(\ga).
 \end{equation}
 The proof of (\ref{var_asymp}) in \cite{Nico_var} is very laborious and computationally intensive. 
 
 The first result of this paper is a functional  version of (\ref{var_asymp}). We achieve this  using a less computational, more robust and    more conceptual technique. One consequence  of this asymptotic estimate  functional  strong law of large numbers  concerning the random measures $\fC[-, F_\ga^N]$, $N\in \bN$. Let us provide some more details.
 
  First some notation. Denote by $\vert-\vert$ the Euclidean norm on $\bR^m$ and by $\vert-\vert_\infty$ the sup-norm on $\bR^m$.  For $\bx_0\in \bR^m$   and $r>0$ we set
 \[
 B_r(x_0):=\lbr\bx\in\bR^m;\;\;\vert\bx\vert\leq r\rbr,\;\; B^\infty_r(x_0):=\lbr\bx\in\bR^m;\;\;\vert\bx\vert_\infty\leq r\rbr.
\]
Clearly  $B_r(x_0)\subset B^\infty_r(x_0)$. 

The function $F^R_\ga$ is $\bZ^m$-periodic  and for $r\in (0,1/2)$ the ball $B^\infty_r(0)$ is contained in the interior of a fundamental domain of the $\bZ^m$-action since $\vert\bx-\by\vert_\infty\leq 2r<1$ and $\vert\vell\vert_\infty\geq 1$, $\forall \vell\in\bZ^m\setminus 0$. This reflects the fact that the injectivity radius of the flat torus $\bT^m =\bR^m/\bZ^m$ is $\leq \frac{1}{2}$ so $B_r(0)$ can be viewed as a geodesic ball. We can now state  the main technical result of this paper.

\begin{theorem}\label{th: main}  Fix an amplitude $\ga$,  a positive integer $m\in \bN$,  a radius $r_0\in (0,1/2)$ and  a  nonnegative  function $f:\bR^m\to\bR$  with support contained in $B_{r_0}(0)$. Then the following hold.

\begin{enumerate} 
\item
\begin{equation}\label{exp_asymp1}
 \lim_{R\to \infty} R^{-m}\bE\lb \fC[ f, F^R_\ga]\rb=\bE\lb \fC[f,\Phi_\ga]\rb =C_m(\ga)\int_{\bR^m} f(\bx)\blam\lb dx\rb.
 \end{equation}

\item There  exists  a constant $V_m(\ga)\geq 0$ that depends only on $m$ and $\ga$ such that
 \begin{equation}\label{var_asymp1}
 \lim_{R\to \infty} R^{-m}\var\lb \fC[ f, F^R_\ga]\rb =V_m(\ga)\int_{\bR^m} f(\bx)^2 d\bx.
 \end{equation}
 \end{enumerate}
 \end{theorem}
 
 If we consider the normalized  random  measures
 \[
 \bar{\fC}_R:=\frac{1}{R^m}\fC[-,F_\ga^R],\;\;R>0
 \]
 then we deduce that for any nonnegative  $f\in C^0_\cpt(\bR^m)$, $\supp f\in B_{r_0}(0)$, we have
 \begin{equation}\label{exp_asymp2}
 \lim_{R\to\infty} \bE\lb  \bar{\fC}_R[f]\rb =C_m(\ga)\int_{\bR^m} f(\bx)\blam\lb dx\rb,
 \end{equation}
 and
 \begin{equation}\label{var_asymp2}
 \var\lb   \bar{\fC}_R[f]\rb  \sim VR^{-m}\;\;\mbox{as $R\to \infty$}.
 \end{equation}
Using finite partitions of unity we deduce from  (\ref{var_asymp2})  that for any nonnegative $f$ with compact support
\[
 \var\lb   \bar{\fC}_R[f]\rb=O\lp R^{-m}\rp\;\;\mbox{as $R\to\infty$}.
 \]
  If  $m\geq 2$, then
 \[
 \sum_{N\in \bN} \frac{1}{N^m}<\infty
 \]
  Borel-Cantelli  and (\ref{var_asymp2}) imply that for any  nonnegative  $f\in C^0_\cpt(\bR^m)$ we have  
  \begin{equation}\label{sSLLN}
  \lim_{N\to \infty}   \bar{\fC}_N[f] = C_m(\ga)\int_{\bR^m} f(\bx)\blam \lb dx\rb\;\;\mbox{a.s. and in $L^2$}.
  \end{equation}
  
  Thus, in the white noise limit ($R\to\infty$), the critical points of $F_\ga^R$ will equidistribute with probability $1$. In the  case $m=1$, this law of large numbers  is proved in the recent  work of  L. Gass \cite[Thm. 1.6]{Gass21}.  

To put (\ref{sSLLN}) in its proper context we need to recall a few facts about the convergence of random measures. For proofs and details we refer to \cite[Chap. 11]{DaVere2} and \cite[Chap. 4]{Kalle_RM}.

 We denote by $\Prob(\bR^m)$ the space of Borel probability measures on $\bR^m$, by $\Meas(\bR^m)$ the space of finite probability measures on $\bR^m$ and by   $\Meas_{\loc}(\bR^m)$ the space of locally finite Borel measures on $\bR^m$, i.e., Borel measures $\mu$ such that $\mu\lb S\rb<\infty$ for any bounded Borel   set $S\subset \bR^m$.  Any  such set   $S$  defines an additive map
\[
L_S:\Meas_\loc(\bR^m)\to\bR,\;\;\Meas_\loc(\bR^m)\ni\mu\mapsto  L_S(\mu)= \mu\lb S\rb \in \bR.
\]
The \emph{vague} topology on the  space $\Meas_\loc(\bR^m)$  is  the smallest topology such that  all the maps $L_S$, $S$ bounded Borel, are  continuous. The space $\Meas_\loc(\bR^m)$ equipped with the vague topology  is a Polish space, i.e., it is separable and the topology is induced by a complete (Prokhorov-like) metric.

A random locally finite  measure on $\bR^m$ is  a measurable map
\[
\mathfrak{M}: \lp \Omega,\eS,\bP\rp\to\Meas_\loc\lp \bR^m\rp,
\]
where $\lp \Omega,\eS,\bP\rp$ is a probability space.  Its distribution is a Borel probability measure $\bP_{\fM}$ on $\Meas_\loc\lp \bR^m\rp$.  It's  mean intensity is the measure $\bar{\fM}$ on  $\bR^m$
\[
\bar{\fM}[S]=\bE\lb \fM[S]\rb 
\]
for any Borel subset  $S\subset \bR^m$.   We say that $\fM$ is locally integrable if  its mean intensity is locally finite.

A sequence of random measures  $\mathfrak{M}_N: \lp \Omega,\eS,\bP\rp\to\Meas_\loc\lp  \bR^m\rp$ is said to converge  vaguely a.s.  to the random measure $\mathfrak{M}$  if
\[
\bP\Lb\lbr \omega;\;\mathfrak{M}_N(\omega)\to\mathfrak{M}(\omega)\;\; \mbox{vaguely in}\;\Meas_\loc(\bR^m)\rbr\Rb=1.
\]
 One can show that  the following statements are equivalent
 \begin{itemize}
 
 \item  $\mathfrak{M}_N\to\mathfrak{M}$ vaguely  a.s..
 
  \item  $\mathfrak{M}_N[S]\to\mathfrak{M}[S]$, a.s., for any bounded Borel $S\subset \bR^m$.
 
  \item  $\mathfrak{M}_N[f]\to\mathfrak{M}[f]$, a.s.,  $\forall f\in C^0_{\cpt}(\bR^m)$.
  
  \end{itemize}
  
  Similarly we say that $\mathfrak{M}_N\to\mathfrak{M}$ vaguely  $L^p$ if the following equivalent conditions hold.
  
   \begin{itemize}

  \item  $\mathfrak{M}_N[S]\to\mathfrak{M}[S]$, $L^p$, for any bounded Borel $S\subset\bR^m$.
 
  \item  $\mathfrak{M}_N[f]\to\mathfrak{M}[f]$, $L^p$,  $\forall f\in C^0_{\cpt}(\bR^m)$.
  
  \end{itemize}

  We can rephrase the  equality (\ref{sSLLN})  as a law as large numbers.   
    
  \begin{corollary}[Strong Law of Large Numbers]\label{cor: main} Suppose that $m\geq 2$. In white noise limit ($N\to\infty$) the \emph{random} measures $\frac{1}{N^m}\fC\lb -, F^N_\ga\rb$ converge vaguely a.s. and $L^2$ to the  \emph{deterministic} measure $C_m(\ga)\blam$. In particular,  for any  bounded Borel subset $S\subset \bR^m$ we have
  \[
\lim_{N\to \infty}  \frac{1}{N^m}\fC\lb S, F^N_\ga\rb= C_m(\ga)\blam \lb S\rb
\]
a.s. and $L^2$.
\qed
  \end{corollary}

    In \cite{Nico_CLT} the second author proved  that in  white noise limit the \emph{random} measures $\fC\lb -, F^N_\ga\rb$ also satisfy a Central Limit Theorem. More precisely, for any $r\in(0,1)$
   \[
   \frac{1}{N^{m/2}}\Lp \fC[B^\infty_{r/2},F_\ga^N]-\bE\lb\fC[B^\infty_{r/2}, F_\ga^N]\rb\Rp
   \]
   converges in distribution to a centered normal random variable with nonzero variance.

   The random  functions $F_\ga^N$ are stationary  and so the random measures $\fC\lb -, F^N_\ga\rb$ are stationary as well, i.e., their distributions are  invariant with respect to the natural action of $\bR^m$ by translations on $\Meas_\loc(\bR^m)$.
   
      As discussed in \cite[Chap.12]{DaVere2} or \cite[Chap.5]{Kalle_RM}, every \emph{stationary}  random locally finite   measure $\fM$ on $\bR^m$  with locally finite mean intensity has an \emph{asymptotic  intensity} $\widehat{\fM}$. This is an  integrable random variable $  \hat{\fM} \in L^1\lp \Meas_\loc(\bR^m),\bP_{\fM}\rp$  with an ergodic meaning, \cite[Sec. 12.2]{DaVere2} or \cite[Sec. 5.4]{Kalle_RM}. More precisely, for any compact convex subset $C\subset \bR^m$ containing the origin in its interior we have   
   \[
   \widehat{\fM}=\lim_{N\to\infty}\frac{1}{\vol\lb NC]}\fM[NC]\;\;\mbox{$\as$ and $L^1$}.
      \]
 The random measure  $\fM=\fC[-, \Phi_{\ga}]$ is stationary and the results of \cite{Nico_SLLN} show that the  asymptotic intensity of $\fC[-,\Phi_\ga]$ is the constant $\widehat{\fC}_{\Phi_\ga}=C_m(\ga)$.  
  We set 
   \[
   B_{R}:=[-R,R]^m=B_{R}^\infty(0).
   \]
   For fixed $R\in\bN$, the random function $\Phi^{R}_\ga$ is $\lp R\bZ\rp^m$-periodic and  we deduce that for any $N\in \bN$  we have
   \[
    \fC[NB_{R}, \Phi_\ga^{N_0}]  =N^m  \fC[NB_{R}, \Phi_\ga^{R}].
    \]
   Hence    
    \[															
    \fC[B_{R}, \Phi_\ga^{R}]=\lim_{N\to \infty}\frac{1}{N^m}\fC[NB_{R}, \Phi_\ga^{R}]=\widehat{\fC}_{\Phi^{R}_\ga} \vol\lb B_{R}\rb,
   \]
  where $\widehat{\fC}_{\Phi^{R}_\ga}$  denotes the asymptotic  intensity of the stationary random measure   $\fC[-, \Phi_\ga^{R}]$. Hence 
   \[
   \widehat{\fC}_{\Phi^{R}_\ga}=\frac{1}{R^m}\fC\lb B_{R}, {\Phi_\ga^{R}}\rb.
   \]
   Corollary \ref{cor: main} shows that 
   \[
   \lim_{N_0\to\infty}\widehat{\fC}_{\Phi^{N_0}_\ga}=\widehat{\fC}_{\Phi_\ga}=C_m(\ga),
   \]
   $\as$ and $L^2$.

To prove  Theorem \ref{th: main} we relate the critical points of $F_\ga^R$ to the critical points of $\Phi_\ga^R$.  The advantage is that $\Phi_\ga^R$ converges in distribution to   the smooth isotropic function $\Phi_\ga$.   To get a  handle on the variance we  express it as in integral  over 
\[
\int_{\bR^m\times \bR^m)\setminus \Delta} \rho^{(2)}_R(\bx,by) d\bx d\by,
\] 
where $\Delta$ is the diagonal and $\rho^{(2)}_R$ is a certain integrand that  blows-up along $\Delta$. 

To  deal with integrability  of $\rho^{(2)}_R$  far away from the diagonal we rely  on  the estimates in Lemma \ref{lemma: key_est} that, roughly speaking,  states that as $R\to\infty$   the covariance kernel $\bsK_\ga$ of $\Phi_\ga$  approximates well the covariance kernel $\bsK_\ga^R$ of $\Phi_\ga^R$  over a  \emph{large} ball, of radius  $\approx R/2$.  The local integrability of $\rho^{(2)}_R$   near the diagonal  is an established fact,  \cite{AnLe23, BMM22, EL}.  It follows from  the blow-up estimate
\[
\sup_{|\bx-\by|<1}|\bx-\by| ^{m-2}\rho^{(2)}_R(\bx,\by)<\infty.
\]
Appendix \ref{s: b} we  refine the strategy in \cite{BMM22, EL} and prove that the above estimate is   \emph{uniform} in $R$.  

Let us say a few things about the organization of the paper. In Section \ref{s: 2} we survey a few  facts about Gaussian measures and Gaussian  random fields and the Kac-Rice formula.  Theorem \ref{th: main} is proved in Section \ref{s: 3}.  We have subdivided this section into several subsections corresponding  to the conceptually different steps in proof of Theorem \ref{th: main}.

 The paper  also includes two technical  appendices. Appendix \ref{s: a}  describes several general conditions guaranteeing  various forms of  ampleness (nondegeneracy) of Gaussian fields. To the best of our knowledge these seem to be  new and we believe they   will find many other applications. Appendix \ref{s: b} contains  explicit  upper bounds  for the variance of the number of zeros  of     a  $C^2$ Gaussian  field   $F$ in a given region $\eR$ in terms of  $C^k$ norms of its covariance kernel of $F$ and the size of $\eR$.   The finiteness of the variance is ultimately due to Fernique's inequality  \cite{Fer} guaranteeing that $\bE\lb \Vert F\Vert_{C^2}^p\rb <\infty$ for all $p\in [1,\infty)$.  The fact that these $L^p$-norms can be controlled by the   covariance kernel of $F$ follows from a result of Nazarov and Sodin \cite{NaSo}.

 \section{Basic Gaussian concepts  and the Kac-Rice formula}\label{s: 2} Let  first  recall a few things about  Gaussian vectors. For details and proofs we refer to \cite{Stroo}.
 
 Suppose  that  $\bsX$  is a finite dimensional  real Euclidean space with inner product $(-,-)$.  An $\bsX$-valued  \emph{Gaussian vector} is a  measurable map $X:  \lp \Omega,\eS,\bP\rp\to\bsX$ such that, for any $\xi\in \bsX$, the random variable $\lp\xi, X\rp$   is Gaussian. Above and in the sequel $\lp \Omega,\eS,\bP\rp$ denotes a probability space. The Gaussian vector $X$ is called  \emph{centered} if   $\lp\xi, X\rp$  has mean zero, $\forall\xi\in\bsX$.
  
 If $X$ is a centered Gaussian  random vector valued in the finite dimensional  Euclidean space $\bsX$, then its distribution is uniquely determined by its \emph{variance}. This  is  the  symmetric nonnegative operator  $\var\lb X\rb:\bsX\to\bsX$ uniquely  determined by  the equality
 \[
 \bE\lb e^{\ii (\xi ,X)}\rb= e^{-\frac{1}{2}(\, \var[X]\xi,\xi\,)},\;\;\forall \xi\in \bsX.
 \]
The Gaussian vector $X$ is called nondegenerate if $\var\lb X\rb$ is nonsingular. 

Suppose that  $X_1,X_2$  are centered  Gaussian  vectors valued in the Euclidean spaces $\bsX_1$ and respectively $\bsX_2$. If $X_1,X_2$ are \emph{jointly Gaussian,} i.e., $X_1\oplus X_2$ is also Gaussian, then the variance of $X_1\oplus X_2$ has the block decomposition
\[
\var\lb X_i\oplus X_2\rb=\left[
\begin{array}{cc}
\var\lb X_1\rb & \cov\lb X_1,X_2\rb\\
\cov\lb X_2,X_1 & \var\lb X_2\rb
\end{array}
\right]
\]
where the \emph{covariance} $\cov\lb X_1,X_2\rb$ is a linear map $ \bsX_2\to \bsX_1$ and 
\[
\cov\lb X_1,X_2\rb= \cov\lb X_1,X_2\rb^*.
\]
If $X_1,X_2$ are jointly Gaussian  and $X_1$ is nondegenerate, then $\bE\lb X_2\cond X_1\rb$, the conditional expectation of $X_2$ given $X_1$,  is an explicit linear  function of $X_1$. Moreover,  for any continuous function $f:\bsX_2\to\bR$ with at most polynomial growth at $\infty$  we have  the \emph{regression  formula} (see \cite[Prop. 12]{AzWs} or \cite[Sec. 2.3.2]{Gass21})
\[
 \bE\lb f(X_2)\lv X_1=0\rb= \bE\lb f(Z)\rb
\]
where $Z=X_2-\bE\lb X_2\cond X_1\rb$  is  a centered Gaussian  vector with variance
 \begin{equation}\label{reg2}
\Delta_{X_2,X_1} =\Var\lb X_2\rb-\cov\lb X_2,X_1]\Var[X_1]^{-1}\cov\lb X_1,X_2\rb.
\end{equation}
 
 Suppose that $\bsU$ and $\bsV$   are finite dimensional  real Euclidean spaces and $\mV\subset \bsV$ is an open set.  For a map $F\in C^k(\mV,\bsU)$ we denote by $F^{(k)}(v)$ the $k$-order differential of $F$ at $v$. It is an element of the vector space $\Sym_k\lp \bsV,\bsU\rp$ consisting of symmetric $k$-linear maps
  \[
  \underbrace{\bsV\times \cdots \times \bsV}_k\to\bsU.
  \]
The \emph{$k$-th jet} of $F$ at $\bv\in \mV$ is the vector
\[
J_kF(\bv):=F(\bv)\oplus F'(\bv)\oplus \cdots \oplus F^{(v)}(\bv).
\]
The  \emph{Jacobian of $F$ at $v\in\mV$} is 
\[
J_F(v)=\sqrt{\det \lp F'(v(F'(v)^*\rp}.
\]
When $\bsU=\bsV$ we have
\[
J_F(v)=\lv \det F'(v)\rv
\]
A $\bsU$-valued \emph{random field }on $\mV$ is a family of random variables
\[
F(v):\lp \Omega,\eS,\bP\rp\to\bsU,\;\;\Omega\ni \omega\mapsto F_\omega(\bv)\in \bsU,\;\;\bv\in\mV.
\]
We will work with measurable random fields , i.e.,  random fields $F$  such that the the map
 \[
 \Omega\times \mV\to \bsU,\;\;(\omega,\bv)\mapsto F_\omega(\bv)
 \]
 is measurable with respect to the  product sigma-algebra on $\Omega\times \mV$. The maps $F_\omega:\mV\to \bsU$ are called the \emph{sample maps} of the random field $F$. The random field is called $C^\ell$ if all its sample maps belong to $C^\ell(\mV,\bsU)$.
 
 The random  field  is called \emph{Gaussian} if,  for any $n\in \bN$, and any $\bv_1,\dotsc,\bv_n\in\mV$, the random vector 
 \[
 \lp F(\bv_1),\dotsc, F(\bv_n)\rp\in \bsU^n
 \] 
 is Gaussian.  In the sequel we will work exclusively with centered Gaussian fields, i.e., $F(\bv)$ is centered, $\forall \bv\in \mV$.  
 
 If $F:\Omega\times \mV\to\bsU$,  the the covariance kernel of $F$ is the  map
 \[
 \eK_F:\mV\times \mV\to\End\lp \bsU\rp,\;\eK_F(\bv_1,\bv_0)=\cov\lb F(\bv_1),F\bv_0)\rb
 \]
Work going back to Kolmogorov shows that if   the covariance kernel of $F$ is sufficiently regular, then so is $F$. More precisely, we will  need the following more precise result, \cite[Appendices A.9-A.11]{NaSo}.
  
\begin{theorem}\label{th: Naz_Sod} Fix $\ell\in \bN_0$ and $\alpha\in (0,1)$. Suppose that $\eV$ is an open subset of $\bR^m$ and $X:\Omega\times \eV\to\bR$ is a  centered Gaussian function with covariance kernel  $\eK_X$. Assume that $\eK\in C^{2\ell+2}\lp \eV\times \eV)$.  Then $X$ admits a $C^{\ell,\alpha}$-modification. Moreover, for every box $B\subset \eV$ and every $p\geq 1$  there exists a constant $C_p=C(B,\eV,\ell,\alpha)$ such that
   \begin{equation}\label{Naz_Sod}
   \bE\lb \Vert X\Vert^p_{C^{\ell,\alpha}(B)}\rb \leq C\lV \eK_X\rV^{p+1}_{C^{2\ell+2}(B\times B)},
  \end{equation}
   where $C^{k,\alpha}$  denotes the spaces of functions that are  $k$ times differentiable and the $k$-th differential is H\"{o}lder continuous with exponent $\alpha$.\qed
   \end{theorem}

  Suppose that $F:\Omega\times\mV\to\bsU$ is a $C^\ell$ Gaussian field.   $F$ is said to be \emph{ample } if for any $\bv\in \mV$ the Gaussian vector $F(\bv)$ is nondegenerate. More generally, if $n$ is a positive integer, then we say that $F$ is \emph{$n$-ample} if, for any pairwise distinct points $\bv_1,\dotsc,\bv_n\in \mV$, the Gaussian vector $F(\bv_1)\oplus\cdots \oplus F(\bv_n)$ is nondegenerate. Let $0\leq k\leq \ell$. We say that $F$ is  \emph{$J_k$-ample} if,  for any $\bv\in \mV$, the  $k$-th jet  $J_kF (\bv)$ is a nondegenerate Gaussian vector.

 We have the following result.   For a proof we refer to \cite[Lemma 11.2.10]{AT} or \cite[Sec. 4]{AAL23}.

 \begin{lemma}[Bulinskaya]\label{lemma: bulin}  Suppose now that $\dim \bsU=\dim\bsV=m$ and $F:\Omega\times \eV\to\bsV$ is an ample  $C^1$, Gaussian random field.  Then, if  $K\subset \mV$ is a compact set of Hausdorff dimension $\leq m-1$, then $F^{-1}(0)\cap K=\emptyset$ $\as$. Moreover
 \[
 \bP\lb \{F(v)=0,\;\;J_F(v)=0\}\rb=0
 \]
\qed
 \end{lemma}
 
 The Kac-Rice  formula plays a central role in this paper. We state below one version of this formula. For a proof we refer to \cite{AAL23} or \cite[Thm.6.2]{AzWs}.

 \begin{theorem}[Local Kac-Rice formula] Suppose that  $F:\Omega\times \eV\to\bsU$ is an ample  $C^1$, Gaussian random field, $m=\dim\bsV=\dim\bsU$. Denote by $p_{F(v)}$ the probability density of the nondegenerate Gaussian vector $F(v)$. For any box $B\subset \mV$ and any nonnegative  continuous function $w\in C(B)$ we set
 \[
 \eZ_B(w,F)=\sum_{\substack{ F(v)=0,\\ v\in B} } w(v)\in [0,\infty].
 \]
 In particular $\eZ(B,F):= \eZ_B(1,F)$ is the number of zeros of $F$ in $B$. Then $\eZ_B(\vfi, F)$,  is measurable, $\as$ finite  and
 \begin{equation}\label{KR}
\bE\lb  \eZ_B(w, F)\rb=\int_B w(v) \rho_{KR}(v) dv
\end{equation}
 where $\rho_{KR}$   is the  \emph{Kac-Rice density}
 \begin{equation}\label{KR_dens_A}
 \rho_{KR}(v):= \bE\lb J_{F}(v)\lv F(v)=0\rb  p_{F(v)}(0).
 \end{equation}\qed
 \end{theorem}

 \begin{corollary}\label{cor: KR_crit} Let  $\eV\subset\bR^m$  an open set.  Suppose that  $\Phi: \eV\to \bR$ a Gaussian random function that is $\as$ $C^2$ and such that  the Gaussian vector $\nabla \Phi(\bv)$ is nondegenerate  for any $\bv\in \eV$. We denote by $p_{\nabla \Phi(v)}$ is probability \emph{density}.  The following hold.
  
  \begin{enumerate}
  
  \item  The random function  $\Phi$ is $\as$ Morse
  
  \item We set
  \begin{equation}\label{rand_crit_meas}
  \fC[-,\Phi]:=\sum_{\nabla F(\bv)=0}\delta_v
  \end{equation}
  Then $\fC[-,\Phi]$ is a random locally finite measure on $\mV$  in the sense of \cite{DaVere2} or \cite{Kalle_RM}.   For any  nonnegative measurable function $\vfi:\mV\to [0,\infty)$ we set
  \[
  \fC[\vfi,\phi]=\int \vfi(v)\fC[dv,\Phi]=\sum_{\nabla \Phi(v)=0}\vfi(v).
  \]
  \item  For any box $B\subset \mV$,   the function $\Phi$ $\as$ has no critical points on $\pa B$ and
  \begin{equation}\label{KR_crit}
  \bE\lb \fC^\Phi[\bsI_B\vfi]\rb=\int_B\bE\lb \vert\det \Hess_\Phi(\bv)\vert \, \rv\,\nabla \Phi(\bv)=0\rb p_{\nabla \Phi(v)}(0) \vfi(\bv)d\bv.
\end{equation}
\end{enumerate}
The quantity 
\begin{equation}\label{KR_crit_dens}
\rho_\Phi:=\bE\lb \vert\det \Hess_\Phi(\bv)\vert \, \rv\,\nabla \Phi(\bv)=0\rb p_{\nabla F(v)}(0)
\end{equation} 
is called the  Kac-Rice (or KR) density of $\Phi$.
 \qed
  \end{corollary}
  
 We conclude  this section with a technical result that will be used several times in the proof of Theorem \ref{th: main}.
 
 Suppose that $\bsV$ is an $m$-dimensional  real Euclidean space with inner product $(-,-)$. Denote by $S_1(\bsV)$  the unit sphere in $\bsV$,  by $\Sym(\bsV)$ the space of symmetric operators $\bsV\to\bsV$, and by $\Sym_{\geq 0}(\bsV)\subset \Sym(\bsV) $ the cone of nonnegative ones. For $A\in  \Sym_{\geq 0}(\bsV)$ we denote  by $\Gamma_A$ the centered  Gaussian measure  on $\bsV$  with variance $A$.
 
 The space $\Sym(\bsV)$ is equipped with an inner product
 \[
 \lp A, B\rp_\op=\tr(AB),\;\;\forall A, B\in \Sym(\bsV).
 \]
 We denote by $\Vert-\Vert_\op$ the associated norm.
 
 We have a natural map $\Sym_{\geq 0}(\bsV)\to \Sym_{\geq 0}(\bsV)$, $A\mapsto  A^{1/2}$.  We will have the following result,  \cite[Prop.2.1]{HeAn}. 
 
 \begin{proposition} For any $\mu>0$ and $\forall A,B\in \Sym_{\geq 0}(\bsV)$,  such that $A^{1/2}+B^{1/2}\geq \mu\one$ we have
  \begin{equation}\label{sq_holder}
\mu \lV A^{1/2}-B^{1/2}\rV_\op \leq \lV A-B\rV^{1/2}_\op. 
 \end{equation}\qed
\end{proposition}
 
 \begin{lemma}\label{lemma: cont_gauss_int} Fix $A_0\in \Sym_{\geq 0}(\bsV)$ such that $A_0^{1/2}\geq \mu_0\one$, $\mu_0>0$. Suppose that $f:\bV\to\bR$ is a  locally Lipschitz function that is  homogeneous of degree $k\geq  1$.  For $A\in \Sym_{\geq 0}(\bsV)$ we set
 \[
 \eI_A(f):=\int_\bsV f(\bv)\bf \Gamma_A\lb d\bv\rb.
 \]
  Then for and $R\geq \Vert A_0\Vert_\op$   there exists a constant $C=C(f,R,\mu_0)>0$  with the following  property:  for any $A\in \Sym_{\geq 0}(\bsV)$ such that $\Vert ,A\Vert_\op\leq R$
 \begin{equation}\label{cont_gauss_int}
 \lv \eI_{A_0}(f)-\eI_A(f)\rv \leq C\Vert A-A_0\Vert^{1/2}\leq C(k,R) \Vert A-A_0\Vert_\op^{1/2}.
 \end{equation}
 In other words, $A\mapsto \eI_A(f)$ is locally H\"{o}lder continuous with exponent  $1/2$ in the  open set $\Sym_{>0}\lp \bsV\rp$.
 \end{lemma}
 
 \begin{proof}The function $f$ is Lipschitz on the ball 
 \[
 B_R(\bsV):=\big\{\, \bv \in \bsV; \Vert \bv\Vert\leq R\,\big\},
 \]
  so there exists $L=L(R)>0$ such that
 \begin{equation}\label{Lipschitz_gauss}
 \lb f(\bu)-f(\bv)\rv \leq L\Vert\bu-\bv\Vert,\;\;\forall \bu,\bv\in B_R(\bsV).
 \end{equation}
 Note that
 \[
 \eI_A(f)=\int_\bsV f\lp A^{1/2}\bv\rp \Gamma_{\one}\lb d\bv\rb,
 \]
 so
 \[
 \lv \eI_{A_0}(f)-\eI_A(f)\rv\leq  \int_\bsV \lv f\lp A^{1/2}\bv\rp -f\lp A_0^{1/2}\bv\rp\rv\;\Gamma_{\one}\lb d\bv\rb
 \]
\[
=\underbrace{\frac{1}{(2\pi)^{m/2}}\left(\int_0^\infty r^{n+k-1} e^{-r^2/2} dr\right)}_{C_{m,k}}\int_{S_1(\bsV)} \lv f\lp A^{1/2}\bv\rp -f\lp A_0^{1/2}\bv\rp\rv\vol_{S_1(\bsV)}\lb d\bv\rb
\]
\[
\stackrel{(\ref{Lipschitz_gauss})}{\leq} C_{m,k}L(R) \int_{S_1(\bsV)} \Vert A^{1/2}-A_0^{1/2}\Vert_\op\vol_{S^1(\bsV)}\lb d\bv\rb\stackrel{(\ref{sq_holder})}{\leq} C(k,R,\mu_0) \Vert A-A_0\Vert_\op^{1/2}.
\]
\end{proof} 
 
 \begin{lemma}\label{lemma: sup_gauss_int}  Suppose that $f:\bsW\to \bR$ is a  continuous  function that is  homogeneous of degree $k\geq  1$.   Set
\[
M(f):=\sup_{\Vert \bw\Vert \leq 1}|f(\bw)|.
\]
Then there exists $C=C(m,k)>0$ such that $\forall A\in \Sym_{\geq 0} (\bsV)$
\begin{equation}
\lv \eI_A(f)\rv\leq \eI_A(|f|)\leq  C(m,k)M(f)\Vert A\Vert^{k/2}_\op.
\end{equation}
\end{lemma}

\begin{proof} Note that
\[
\sup_{\Vert \bu\Vert \leq R} |f(\bu)|=M(f)R^k.
\]As in the proof of Lemma \ref{lemma: cont_gauss_int}  we have
\[
\eI_A(|f|)= \int_{\bsW} f\lp A^{1/2}\bw\rp \bGamma_{\one}\lb d\bw\rb
\]
\[
=\underbrace{\frac{1}{(2\pi)^{m/2}}\left(\int_0^\infty r^{m+k-1} e^{-r^2/2} dr\right)}_{=: C_{m,k}}\int_{S_1(\bsW)} \lv f\lp A^{1/2}\bw\rp\rv  \vol_{S_1(\bsV)}\lb d\bw\rb
\]
($ \Vert A^{1/2}\bw\Vert\leq \Vert A^{1/2}\Vert_\op \Vert \bw\Vert$)
\[
\leq C_{m,k} M(f)\Vert A^{1/2}\Vert^k_\op \vol\lb S_1(\bsV)\rb=C(m,k) M(f)\Vert A\Vert^{k/2}_\op.
\]
\end{proof}

\begin{corollary}\label{cor: sup_gauss_int}  Suppose that $f:\bsW\to \bR$ is a  continuous  function that is  homogeneous of degree $k\geq  1$. Suppose that  $A,B\in \Sym_{\geq 0}(\bsW)$ and $B\leq A$. Then
\begin{equation} 
\lv \eI_B(f)\rv\leq \eI_B(|f|)\leq  C(m,k)M(f)\Vert A\Vert^{k/2}_\op
\end{equation}
\end{corollary}

\begin{proof} Indeed, $0\leq B\leq A\implies \Vert B\Vert_\op\leq \Vert A\Vert_\op$.
\end{proof}

 \section{Proof of Theorem \ref{th: main}}\label{s: 3}  We split the proof of Theorem \ref{th: main} into several conceptually distinct parts.
 
 \subsection{The key estimate}The following technical result will play a key role.
  
 \begin{lemma}\label{lemma: key_est} Fix $r_0\in (0,1)$ and  a box $B=B^\infty_{r_0/2}(0)=[-r_0/2,r_0/2]^m$. Then the following hold.
  
  \begin{enumerate}
  
  \item For any $\ell\in \bN_0$ and any $p>m$ there  exists $C=C(p,m,\ell,\ga)>0$, independent of $R$,  such that, $\forall R>2$
  \[
  \lV \bsK^R_\ga-\bsK_\ga\rV_{C^\ell(RB)}\leq CR^{-p}
  \]
  \item For any $\ell\in \bN_0$ and any $p>m$ there  exists $C=C(p,m,\ell,\ga)>0$, independent of $R$,  such that, $\forall R>2$, $\forall \bx,\by\in RB$
  \[
  \lv D^\ell \bsK^R_\ga(\bx-\by)\rv \leq\frac{C}{\lp 1+|\bx-\by|_\infty\rp^p}.
  \]
  \end{enumerate}
  \end{lemma}
  
  \begin{proof} (i)  Denote by $\eT_{R\vk} \bsK_\ga $ the translate
\[
\eT_{R\vk}\bsK_\ga(\bx):=   \bsK\lp \bx-R\vk\rp.
\]
 We have
  \[
  \bsK_\ga^R(\bx)-\bsK_\ga(\bx)=\sum_{\vk\in\bZ^m\setminus 0}  \eT_{R\vk}\bsK_\ga \lp \bx\rp.
  \] 
 Now observe that $\forall R>0$,   $\forall \bx\in RB$, and any $\vk\in \bZ^m\setminus 0$ we have
 \[
\lv \bx -R\vk\rv_\infty\geq N\lv\vk\rv_\infty-\lv\bx\rv_\infty\geq R\lp \lv\vk\rv_\infty-r_0/2 \rp .
\]
 Since $\bsK_\ga$  and all its derivatives are Schwartz functions we deduce that   for any $p>m$, and any $\vk\in\bZ^m\setminus 0$
\[
\lV \eT_{R\vk}\bsK_\ga\rV_{C^\ell(NB)}\leq C(p,m,\ell,\ga) R^{-p}\lp \lv \vk\rv_\infty-r_0/2\rp^{-p}.
\]
The last expression is well defined since $r<1\leq \lv\vk\rv_\infty$ for any $\vk\in \bZ^m\setminus 0$. Hence 
\[
 \lV \bsK^R_\ga-\bsK_\ga\rV_{C^\ell(NB)}\leq C(p,m,\ell,\ga)R^{-p}\sum_{\vk\in\bZ^m\setminus 0} \lp \lv \vk\rv_\infty-r_0/2\rp^{-p}
 \]
 The above series is convergent since $p>m$.
 
 \smallskip
 
 \noindent (ii)  Note that  $\forall \bx,\by\in RB$ we have $\lv\bx-\by\rv_\infty\leq Rr_0$. Set $\bz:=\bx-\by$.  We discuss only the case $\ell=0$.   The general case can  be reduced to this case by taking partial derivatives.
 
  Using (i) we deduce  that  
 \[
C=\sup_{R}\sup_{\lv \bz\rv_\infty<r_0} \lv \bsK^R_\ga(\bz)\rv< \infty
\]
and thus, $\forall R\geq 2$, $\forall \lv\bz\rv_\infty<r_0$,
\[
 \lv \bsK^R_\ga(\bz)\rv< \frac{C\lp 1+r_0\rp^p}{\lp 1+\lv\bz\rv_\infty\rp^p}.
 \]
 Assume now that $\lv\bz\rv_\infty\geq r_0$. We have
 \[
 \bsK^R_\ga(\bz)=\bsK_\ga (\bz)+\sum_{\vk\in \bZ^m\setminus 0}  \eT_{R\vk}\bsK_\ga\lp \bz\rp,
  \]
 and thus,
 \[
 \lv  \bsK^R_\ga(\bz)\rv\leq  \lv \bsK_\ga (\bz)\rv+\sum_{\vk\in \bZ^m\setminus 0}  \lv \eT_{R\vk}\bsK_\ga\lp \bz\rp\rv.
 \]
 Since $\bsK_\ga(\bx)$ is Schwartz we deduce that there exists  a constant $C=C(p,\ga)$ such that
 \[
 \frac{C_p}{\lp 1+\lv \bz\rv_\infty\rp^p}+C_p\sum_{\vk\in \bZ^m\setminus 0}   \frac{1}{\lp 1+\lv \bz-R\vk\rv_\infty\rp^p}.
 \]
 We have  $\lv \bz\rv_\infty\leq Rr_0$ and
  \[
  \lv \bz-Z\vk\rv_\infty\geq \lv \bz\rv_\infty\Lp  \frac{R\lv\vk\rv_\infty}{\lv \bz\rv_\infty} -1\Rp\geq \lv \bz\rv_\infty\Lp  \frac{1}{r_0}\lv\vk\rv_\infty -1\Rp.
  \]
  Thus
  \[
  \sum_{\vk\in \bZ^m\setminus 0}   \frac{1}{\lp 1+\lv \bz-R\vk\rv_\infty\rp^p}\leq \lv\bz\rv_\infty^{-p}\;\underbrace{   \sum_{\vk\in \bZ^m\setminus 0}  \Lp  \frac{1}{r_0}\lv\vk\rv_\infty -1\Rp^{-p}}_{<\infty}.
  \]
  \end{proof}
  
  \subsection{An integral formula} Set  $B=B^\infty_{r_0/2}(0)$, $f_R(\bx)=f(\bx/R)$
  \[
  Z^R[f]=\fC[f, F_\ga^R]=\fC[f_R,\Phi_\ga^R],\;\;Z[f]=\fC[f,\Phi_\ga].
  \]
 Denote by $\rho_\ga^R$ the Kac-Rice density of $\Phi_\ga^R$   and by $\rho_\ga$ the Kac-Rice density of $\Phi_\ga$; see (\ref{KR_crit_dens}).  Since both $\Phi_\ga^R$ and $\Phi_\ga$ are stationary random functions we deduce that both $\rho_\ga^R$ and $\rho_\ga$ are constant functions. We set and $C_m(\ga):=\rho_\ga(0)$.   For an explicit description of $C_m(\ga)$ we refer to \cite{Nico_SLLN} .
 
 The covariance functions $\bsK_\ga^R(\bz)$ and $\bsK_\ga(\bz)$ are even so the odd order derivatives of these functions  vanish at $0$. This implies that  the Gaussian  vectors $\Hess_{\Phi_\ga^R}(0)$ and $\nabla\Phi_\ga^R(0)$ are independent.  A similar phenomenon is true for $\Phi_\ga$. Thus, the  conditional  expectations in (\ref{KR_crit_dens}) are usual  expectations. Using  Lemma \ref{lemma: cont_gauss_int} and  Lemma \ref{lemma: key_est}(i)  we deduce that for any $\bx\in \bR^m$
 \begin{equation}\label{KR_dens_N}
 \sup_{\bx\in RB}\lv\rho_\ga^R(\bx)-\rho_\ga(\bx)\rv=\lv\rho_\ga^R(0)-\rho_\ga(0)\rv= O\lp  R^{-\infty}\rp,
 \end{equation}
 where $O(N^{-\infty})$ is short-hand for $O(N^{-p})$, $\forall p>0$. We deduce that
 \[
 R^{-m} \lp \bE\lb Z^R[f]\rb-\bE\lb Z[f]\rb\rp=R^{-m}\int_{RB} f_R(\bx)\lp \rho_\ga^R(0)-\rho_\ga(0)\rp d\bx
 \]
\[ 
=\int_Bf(\by) \lp \rho_\ga^R(0)-\rho_\ga(0)\rp d\by= O\lp R^{-\infty}\rp.
\]
On the other hand
 \[
\bE\lb Z[f]\rb =C_m(\ga)\int_{\bR^m}f(\bx) d\bx.
\]
We need to introduce some notation. Set

 \begin{itemize} 
 
 \item $\Phi_\ga^\infty=\Phi_\ga$.
 
 \item For any $R\in (0,\infty]$ we  define
\[
\hPhi^R,\;\hPhi:\bR^m\times \bR^m\to\bR,
\]
\[
\hPhi^R(\bx,\by)=\Phi_\ga^R(\bx)+\Phi_\ga^R(\by),\;\;\hPhi(\bx,\by)=\Phi_\ga(\bx)+\Phi_\ga(\by),
\]
\[
\hat{\fC}^R:=\fC[-,\hPhi^R_\ga],\;\; \hh_R(\bx,\by):=\Hess_{\hPhi^R}(\bx,\by),\;\;H_R(\bx):=\Hess_{\Phi^R_\ga}(\bx).
\]
\item Choose  an independent copy $\Psi^R_\ga$ of $\Phi^R_\ga$ and  for $R\in (0,\infty]$ set
 \[
 \tphi^R(\bx,\by):=\Phi^R_\ga(\bx)+\Psi^R_\ga(\by),\;\;\tH_R(\bx,\by):=\Hess_{\tphi^R}(\bx,\by),
 \]
 $\tilde{\fC}^R=\fC[-, \tphi^N]$. 
 
 \item For $R\in (0,\infty)$ define 
 \[
 f_R^{\boxtimes 2}:\bR^m\times \bR^m\to\bR, \;\;f_R^{\boxtimes}(\bx,\by)=f_R(\bx)f_R(\by)
 \]
  and set $\Vert f\Vert:=\Vert f\Vert_{C^0(\bR^m)}$. 
 
 \item Set
 \[
 \eX= \bR^m\times \bR^m\setminus \Delta=\big\{\, (\bx,\by)\in \bR^m\times\bR^m;\;\;\bx\neq \by\,\big\}.
 \]
 \end{itemize}
 Observe that  the random function on  $ \tphi^R(\bx,\by)$ is \emph{stationary} with respect to the action of $\bR^{2m}$ on itself by translations. 
 
 We have 
 \[
 \hat{\fC}^R[\bsI_{\eX} f_R^{\boxtimes 2}]=\sum_{\substack{\nabla\Phi_\ga^R(\bx)=\nabla\Phi_\ga^R(\by)=0,\\ \bx\neq \by}} f_R(\bx)f_R(\by)=Z^R[f]^2-Z^R[f^2].
 \]
 Bulinskaya's lemma implies that  
 \[
 \bP\lb \exists \bx:\;\;\nabla\Phi_\ga(\bx)=\nabla\Psi_\ga(\bx)=0\rb=0
 \]
 and we deduce
  \[
 \tilde{\fC}^R[\bsI_{\eX} f_R^{\boxtimes 2}]=\sum_{\substack{\nabla\Phi_\ga^R(\bx)=\nabla\Psi_\ga^R(\by)=0,\\ \bx\neq \by}} f_R(\bx)f_R(\by)
 \]
 \[
 = \sum_{\nabla\Phi_\ga^R(\bx)=\nabla\Psi_\ga^R(\by)=0} f_R(\bx)f_R(\by)= \fC[f,\Phi_\ga^R]\rb \fC[f,\Psi_\ga^R],\;\;\as.
 \]
 Hence
 \[
 \bE\lb \fC[f,\Phi_\ga^R] \fC[f,\Psi_\ga^R]\rb=\bE\lb   \fC[f,\Phi_\ga^R]\rb\,\cdot\, \bE\lb   \fC[f,\Psi_\ga^R]\rb=\bE\lb   \fC[f,\Phi_\ga^R]\rb^2
 \]
 so that
 \[
\bE\lb  \hat{\fC}^R[\bsI_{\eX} f_R^{\boxtimes 2}]\rb -\bE\lb  \tilde{\fC}^R[\bsI_{\eX} f_R^{\boxtimes 2}]\rb= \underbrace{\bE\lb Z^R[f]^2\rb-\bE\lb Z^R[f]\rb^2}_{=\var\lb Z^R[f]\rb}-\bE\lb Z^R[f^2]\rb
\]
We have seen that 
\[
\lim_{R\to\infty}R^{-m}\bE\lb Z^R[f^2]\rb=C_m(\ga)\int_{\bR^m}f^2(\bx) d\bx
\]
so we have to show that
\begin{equation}\label{I(R)}
I(R):=\bE\lb  \hat{\fC}^R[\bsI_{\eX} f_R^{\boxtimes 2}]\rb -\bE\lb  \tilde{\fC}^R[\bsI_{\eX} f_R^{\boxtimes 2}]\rb\sim Z_m(\ga)R^{m}\int_{\bR^m}f^2(\bx) d\bx\;\;\mbox{as $R\to\infty$}
\end{equation}
for some constant  $ Z_m(\ga)\in \bR$ that depends  only on $m$ and $\ga$.

 According to Corollary \ref{cor: hbar_2amp}, there exists $R_0>0$ such that for $R\geq R_0$, the gradient $\nabla\Phi_\ga^R$ is $2$-ample  and  $\Phi_R$  is $J_1$-ample so, for $R\geq R_0$  the gradient $\nabla\hPhi^R(\bx,\by)$ is nondegenerate  for any $\bx\neq \by$ and the random vector $\lp \Phi_\ga^R(\bx),\nabla \Phi_\ga^R\rp$  is nondegenerate  for any $\bx\in \bR^n$.   As shown in \cite{Nico_SLLN} this is true also for $R=\infty$, where we recall that $\Phi_\ga^\infty=\Phi_\ga$. 
 
 We can apply the Kac-Rice  formula and we deduce  that for any $R>R_0$ we have
 \begin{equation}\label{KR_cov_1aa}
\begin{split}
\bE\lb  \hat{\fC}^R[\bsI_{\eX} f_R^{\boxtimes 2}]\rb\hspace{5cm}\\
=  \int_{\bR^m\times \bR^m\setminus\Delta } \underbrace{\bE\lb \vert\det \hh_R(\bx,\by)\vert \rv \nabla \hphi^R(\bx,\by)=0\rb p_{\nabla\hphi^R(\bx,\by)}(0)}_{=\hrho_R(\bx,\by)} f_R^{\boxtimes 2}(\bx,\by)\blam\lb d\bx d\by\rb.
\end{split}
\end{equation}
The gradient $\nabla\tphi^R(\bx,\by)$ is nondegenerate for any $\bx,\by$ and invoking Kac-Rice again we obtain
\begin{equation}\label{KR_cov_1aat}
\begin{split}
\bE\lb  \tilde{\fC}^R[\bsI_{\eX} f_R^{\boxtimes 2}]\rb\hspace{5cm}\\
=  \int_{\bR^m\times \bR^m\setminus\Delta } \underbrace{\bE\lb \vert\det \tH_R(\bx,\by)\vert \rv \nabla \tphi^R(\bx,\by)=0\rb p_{\nabla\tphi^R(\bx,\by)}(0)}_{=\trho_R(\bx,\by)} f_R^{\boxtimes 2}(\bx,\by)\blam\lb d\bx d\by\rb.
\end{split}
\end{equation}
The function  $\trho_R(\bx,\by)$ is independent of $\bx,\by$ since the random function $\tphi^R$  is stationary. Thus
\begin{equation}\label{KR_cov_2_per}
\begin{split}
I(R)=\int_{\eX}\lp \hrho_R(\bx,\by)-\trho_R(\bx,\by)\rp f_R(\bx)f_R(\by)\blam\lb d\bx d\by\rb\\
=\int_{\substack{|\bx|,\,|\by|\leq Rr_0/2,\\ \bx\neq\by}}\lp \hrho_R(\bx,\by)-\trho_R(\bx,\by)\rp f_R(\bx)f_R(\by)\blam\lb d\bx d\by\rb.
\end{split}
 \end{equation}
 Let us observe that for any $\bx\neq \by$ we have
 \[
 \lim_{R\to\infty} \lp \hrho_R(\bx,\by)-\trho_R(\bx,\by)\rp= \lp \hrho_\infty(\bx,\by)-\trho_\infty(\bx,\by)\rp.
 \]
 Moreover
 \[
 \lim_{R\to\infty} f_R(\bx)= f(0)
 \]
 uniformly on compacts.

\subsection{Off-diagonal  behavior}  Note that
\[
\var\lb \tH_R(\bx, \by)\rb =\left[\begin{array}{cc}
\Var\lb  H_R(\bx)\rb & 0\\
0 & \var\lb H_R(\by)\rb
\end{array}
\right].
\]
For every $\bz\in \bR^m$ we set 
 \[
 T_R(\bz):=\sum_{|\alpha|\leq 4}\lv \pa^\alpha \bsK^R_\ga(\bz)\rv.
 \]
   Lemma \ref{lemma: key_est}(ii) shows that for every $p>0$ there exists $C_p=C_p(\ga, m,r)>0$ such that, $\forall R$, $\forall \lv \bz\rv_\infty<Nr$
   \begin{equation}\label{T_N}
  \forall N,\;\;\forall \lv \bz\rv_\infty<Rr_0,\;\; T_R(z)\leq  C_p\lp 1+\lv \bz\rv_\infty\rp^{-p}.
   \end{equation}
We want to emphasize that \emph{$C_p$ is  independent of $R$}.

  Observe next that
 \[
 \var\lb \nabla \tphi^R(\bx,\by)\rb=\left[
 \begin{array}{cc}
 \var\lb \nabla\Phi^R_\ga(\bx)\rb  & 0\\
 0 & \var\lb \nabla\Phi^R_\ga(\by)\rb
 \end{array}
 \right],
 \]
 is independent of $\bx$ and $\by$.
 \[
  \var\lb \nabla \hphi^R(\bx,\by)\rb=\left[
 \begin{array}{cc}
 \var\lb \nabla\Phi^R_\ga(\bx)\rb  & \cov\lb \nabla\Phi^R_\ga(\bx),\nabla\Phi^R_\ga(\by)\rb\\
 & \\
 \cov\lb \nabla\Phi^R_\ga(\by),\nabla\Phi^R_\ga(\bx)\rb & \var\lb \nabla\Phi^R_\ga(\by)\rb
 \end{array}
 \right]
 \]
 \[
 = \var\lb \nabla \tphi^R(\bx,\by)\rb+ \underbrace{\left[
 \begin{array}{cc}
0& \cov\lb \nabla\Phi^R_\ga(\bx),\nabla\Phi^R_\ga(\by)\rb\\
 & \\
 \cov\lb \nabla\Phi^R_\ga(\by),\nabla\Phi^R_\ga(\bx)\rb &0 
 \end{array}
 \right]}_{=: \eE^R_\nabla(\bx,\by)}.
 \]
 Hence
 \begin{equation}\label{KR_cov_3a_per}
 \lV  \var\lb \nabla \hphi^R(\bx,\by)\rb-  \var\lb \nabla \tphi^R(\bx,\by)\rb\rV_\op=\Vert \eE^R_\nabla(\bx,\by)\Vert_\op=O\lp T_R(\bx-\by)\rp,
\end{equation}
where $\Vert-\Vert_\op$ denotes the operator norm. Above and in the sequel, the \emph{constant implied by the Landau symbol $O$ is independent of $R$ as long as $\bx,\by\in RB$.}  In particular
 \begin{equation}\label{invargrad}
\begin{split}
\var\lb   \nabla\hphi^R(\bx,\by)\rb^{-1}=\Lp  \var\lb   \nabla\tphi^R(\bx,\by)\rb+\eE^R_\nabla(\bx,\by)\Rp^{-1}\\
=\Lp \one + \var\lb   \nabla\tphi^R(\bx,\by)\rb^{-1}\eE^R_\nabla(\bx,\by)\Rp^{-1}\var\lb   \nabla\tphi^R(\bx,\by)\rb^{-1}.
\end{split}
\end{equation}
As shown in \cite{Nico_SLLN},  there exists an explicit positive constant $d_m$ such that 
\[
\Var\lb\nabla\Phi_\ga(\bx)\rb=d_m\one_m, \;\;\forall \bx.
\]
 Then  $\var\lb \nabla\Phi^R_\ga(\bx)\rb=  \var\lb \nabla\Phi^R_\ga(0)\rb$, $\forall \bx\in \bR^m$ and
 \[
 \var\lb \nabla\Phi^R_\ga(0)\rb= d_m\one_m+ O\lp R^{-\infty}\rp.
 \]
 The variance  $\var\lb \nabla\tphi^R(\bx,\by)\rb$ is independent of $\bx$ and $\by$ and
  \begin{equation}\label{gradtn}
\var\lb \nabla\tphi^R(\bx,\by)\rb= \var\lb \nabla\Phi^R_\ga(0)\rb\oplus  \var\lb \nabla\Phi^R_\ga(0)\rb= d_m\one_{2m}+ O\lp R^{-\infty}\rp.
\end{equation}
From (\ref{invargrad}) and (\ref{gradtn}) we conclude that  there exists $C_0>0$, independent of $R>R_0$,  such that 
\[
\begin{split}
\Vert\var\lb   \nabla\tphi^R(\bx,\by)\rb^{-1}\eE^R_\nabla(\bx,\by)\Vert_\op<\frac{1}{2},\;\;\forall \bx,\by\in RB,\;\;\vert\bx-\by\vert_\infty>C_0,
\end{split}
\]
and thus
\begin{equation}\label{var_grad_per}
 \begin{split}
 \lV  \var\lb \nabla \hphi^R(\bx,\by)\rb^{-1}-  \var\lb \nabla \tphi^R(\bx,\by)\rb^{-1}\rV_\op\\
 =O\lp T_R(\bx-\by)\rp,\;\;\forall \bx,\by\in RB,\;\;\vert\bx-\by\vert_\infty>C_0.
 \end{split}
 \end{equation}
 
 Note that since $\Phi^R_\ga$ is stationary,  $\var\lb \tH_R(\bx, \by)\rb$ is \emph{independent} of $\bx$ and $\by$. 
 \[
\var\lb \hh_R(\bx, \by)\rb =\left[\begin{array}{cc}
\Var\lb  H_R(\bx)\rb & \cov\lb H_R(\bx),H_R(\by)\rb\\
& \\
\cov \lb H_R(\by), H_R(\bx)\rb & \var\lb H_R(\by)\rb
\end{array}
\right]
\]
\[
= \var\lb \tH_R(\bx, \by)\rb+ \underbrace{\left[\begin{array}{cc}
0 & \cov\lb H_R(\bx),H_R(\by)\rb\\
& \\
\cov \lb H_R(\by), H_R(\bx)\rb & 0
\end{array}
\right]}_{=: \eE_H^R(\bx,\by)}.
\]
  We deduce 
 \begin{equation}\label{KR_cov_3aa_per}
 \lV  \var\lb \hh_R(\bx, \by)\rb-\var\lb \tH_R(\bx, \by)\rb\rV_\op=\Vert \eE_H^R(\bx,\by)\Vert_\op = O\lp T_R(\bx-\by)\rp.
 \end{equation}
 We denote by  $ \tH_R(\bx,\by)^\flat$ the Gaussian random  matrix
 \[
\tH_R(\bx,\by)^\flat=\tH_R(\bx,\by)-\bE\lb \tH_R(\bx,\by)\cond  \nabla\tphi^R(\bx,\by)\rb.
\]
 We define  $ \hh_R(\bx,\by)^\flat$ similarly
\[
 \hh_R(\bx,\by)^\flat=\hh_R(\bx,\by)-\bE\lb \hh_R(\bx,\by)\cond  \nabla\hPhi^R(\bx,\by)\rb.
 \]The distributions of   $ \tH_R(\bx,\by)^\flat$ and $ \hh_R(\bx,\by)^\flat$ are determined by the Gaussian  regression formula (\ref{reg2}). We have
\[
\cov\lb  \hh_R(\bx,\by), \nabla\hphi^R(\bx,\by)\rb=\left[
\begin{array}{cc}
\cov\lb H_R(\bx),\nabla\Phi^R_\ga(\bx)\rb &  \cov\lb H_R(\bx),\nabla\Phi^R_\ga(\by)\rb\\
&\\
\cov\lb H_R(\by),\nabla\Phi^N_\ga(\bx)\rb &  \cov\lb H_R(\by),\nabla\Phi^R_\ga(\by)\rb
\end{array}
\right]
\]
\[
= \left[
\begin{array}{cc}
\cov\lb H_R(0),\nabla\Phi^R_\ga(0)\rb &  \cov\lb H_R(\bx),\nabla\Phi^R_\ga(\by)\rb\\
&\\
\cov\lb H_R(\by),\nabla\Phi^R_\ga(\bx)\rb &  \cov\lb H_R(0),\nabla\Phi^R_\ga(0)\rb
\end{array}
\right].
\]
The covariance $\cov\lb H_R(0),\nabla\Phi^R_\ga(0)\rb$ involves only third order partial derivatives of $\bsK^N_\ga$ at $0$, and these are all trivial since $\bsK^R_\ga$ is an even function. Hence 
\[
\cov\lb  \hh_R(\bx,\by), \nabla\hphi^R(\bx,\by)\rb= \left[
\begin{array}{cc}
0 &  \cov\lb H_R(\bx),\nabla\Phi^R_\ga(\by)\rb\\
&\\
\cov\lb H_R(\by),\nabla\Phi^R_\ga(\bx)\rb & 0
\end{array}
\right].
\]
Similarly
\[
\cov\lb  \tH_R(\bx,\by), \nabla\tphi^R(\bx,\by)\rb= \left[
\begin{array}{cc}
\cov\lb H_R(\bx),\nabla\Phi^R_\ga(\bx)\rb &  0\\
&\\
0 &  \cov\lb H_R(\by),\nabla\Phi^R_\ga(\by)\rb
\end{array}
\right]=0.
\]
Lemma \ref{lemma: key_est}(ii) implies that
 \[
 \begin{split}
\lV \cov\lb  \tH_R(\bx,\by), \nabla\tphi^R(\bx,\by)\rb\rV_\op =O\lp T_R (\bx-\by)\rp,\\
\lV \cov\lb  \hh_R(\bx,\by), \nabla\hphi^R(\bx,\by)\rb\rV_\op=O\lp T_R(\bx-\by)\rp.
\end{split}
\]
Since $\var\lb  \nabla\tphi^N(\bx,\by)\rb$  and we deduce from  the regression formula  (\ref{reg2})  that
 \[
  \Var\lb  \tH_R(\bx,\by)\rb= \Var\lb  \tH_R(\bx,\by)^\flat\rb+ O\lp T_R( \bx-\by)\rp,
 \]
 \[
\var\lb \hh_R(\bx, \by)\rb= \Var\lb  \tH_R(\bx,\by)^\flat\rb+ O\lp T_R( \bx-\by)\rp.
\]
 The regression formula  (\ref{reg2}) shows that
\[
 \begin{split}
 \Var\lb \hh_R(\bx,\by)^\flat\rb=\var\lb \hh_R(\bx,\by)\rb\hspace{3cm}&\\
  -\cov\lb  \hh_R(\bx,\by), \nabla\hphi^R(\bx,\by)\rb \var\lb   \nabla\hphi^R(\bx,\by)\rb^{-1}\cov  \nabla\hphi^R(\bx,\by),\hh_R(\bx,\by)\rb. &
  \end{split}
 \]
 \[
 \begin{split}
 =\var\lb \tH_R(\bx,\by)^\flat\rb+O\lp T_R(\bx-\by)\rp\hspace{1cm}&\\
  -\cov\lb  \hh_R(\bx,\by), \nabla\hphi^R(\bx,\by)\rb \var\lb   \nabla\hphi^R(\bx,\by)\rb^{-1}\cov\lb  \nabla\hphi^R(\bx,\by),\hh^R(\bx,\by)\rb. &
  \end{split}.
  \]
Since $\cov\lb  \hh_R(\bx,\by), \nabla\hphi^R(\bx,\by)\rb=O\lp T_R(\bx-\by)\rp$  we deduce  from (\ref{gradtn}) and (\ref{var_grad_per}) that there exists $C_1>0$, independent of $R>R_0$,  such that 
\[
\begin{split}
\cov\lb  \hh_R(\bx,\by), \nabla\hphi^R(\bx,\by)\rb \var\lb   \nabla\hphi^R(\bx,\by)\rb^{-1}\cov\lb  \nabla\hphi^R(\bx,\by),\hh_R(\bx,\by)\rb\\
=O\lp T_R(\bx,\by)\rp,\;\;\forall \bx,\by\in RB,\;\;\vert\bx-\by\vert_\infty>C_1,
\end{split}
\]
and  thus
\[
\begin{split}
\lV \Var\lb \hh_R(\bx,\by)^\flat\rb-\var\lb \tH_R(\bx,\by)^\flat\rb\rV_\op\hspace{3cm}\\
 = O\lp T_R(\bx-\by)\rp,\;\;\forall \bx,\by\in RB,\;\;\vert\bx-\by\vert_\infty>C_2=\max(C_0,C_1).
 \end{split}
\]
Since $\var\lb  \tH_R(\bx,\by)\rb=\var\lb H_R(0)\rb\oplus \var\lb H_R(0)\rb$  we deduce  from Lemma \ref{lemma: key_est}(i) that there exists $\mu_0>0$ such that
\[
\var\lb \tH_R(\bx,\by)^\flat\rb\geq \mu_0\one,\;\;\forall R\geq R_0.
\]
Note also that (\ref{KR_cov_3a_per}) implies that there exists $C_3>0$, independent of $R>R_0$, such that
\[
\sup_{\substack{\bx, \by\in RB\\ \vert\bx- \by\vert_\infty>C_3}}\Vert \var\lb \hh_R(\bx,\by)^\flat\rb\Vert_\op =O(1).
\]
Lemma  \ref{lemma: cont_gauss_int} implies that 
\begin{equation}\label{KR_cov_5_per}
\Lv \bE\lb |\det \hh_R(\bx,\by)^\flat|\rb- \bE\lb |\det \tH_R(\bx,\by)^\flat|\rb\Rv= O\lp T_R(\bx-\by)^{1/2}\rp.
\end{equation}
Using  (\ref{var_grad_per}) we deduce that there exists $C_4>0$, independent of $R>R_0$, such that
\begin{equation}\label{KR_cov_5a_per}
\begin{split}
\Lv p_{ \nabla\hphi^R(\bx,\by)}(0)- p_{\nabla\tphi^R(\bx,\by)}(0)\rv\hspace{3cm}\\
=\frac{1}{(2\pi)^{m/2}}\Lv \det \var\lb   \nabla\hphi^R(\bx,\by)\rb^{-1}-\det\var\lb   \nabla\tphi^R(\bx,\by)\rb^{-1}\Rv\\
=O\lp  T_R(\bx-\by)\rp,\;\;\forall \bx,\by\in RB,\;\;\vert\bx-\by\vert_\infty>C_4.
\end{split}
\end{equation}
We can now estimate the right-hand-side of (\ref{KR_cov_2_per}).  For any $\bx,\by\in RB$ 
 \[
  O\lp T_R(\bx-\by)\rp\stackrel{(\ref{T_N})}{=} O\lp \lv\bx-\by\rv_\infty^{-p/2}\rp,\;\;\forall p>0.
\]
Using (\ref{var_grad_per}),  (\ref{KR_cov_3aa_per}),   (\ref{KR_cov_5_per}) and (\ref{KR_cov_5a_per})  that
 we conclude that there exists $C_5>1$, independent of $R>R_0$ such that, for any $p>m$,
 \begin{equation}\label{tildeKR}
\forall \bx,\by\in RB,\;\;\vert\bx-\by\vert_\infty>C_5,\;\;\lv \underbrace{\hrho_R(\bx,\by)-\trho_R(\bx,\by)}_{=\Delta_R(\bx,\by)}\rv=  O\lp \lv\bx-\by\rv_\infty^{-p/2}\rp. 
 \end{equation}
 Since the random function $\Phi_\ga^R$ is stationary, we deduce that for any $\bx,\by,\bz\in \bR^m$  such that $\bx\neq \by$ we have
 \[
  \Delta_R(\bx+\bz,\by+\bz)=\Delta_R(\bx,\by)
  \]
  so $\hrho_R(\bx,\by)$, $\trho_R(\bx,\by)$ and  $\Delta_R(\bx,\by)$ depend only on $\by-\bx$.

  \subsection{Conclusion} Assume now that $\bx,\by\in RB$ and  $\vert\bx-\by\vert_\infty\leq C_5$.  Denote by $\widehat{\eX}$ the  radial-blowup of $\bR^m\times \bR^m$ along the diagonal.  It is diffeomorphic to the product  $\bR^m \times S^{m-1}\times [0,\infty)$.

  Choose new orthogonal  coordinates $(\xi,\eta)$  given  by 
  \[
  \xi=\bx+\by,\;\;\eta=\bx-\by\,\Llra\, \bx=\frac{1}{2}(\xi+\eta),\;\;\by=\frac{1}{2}(\xi-\eta)
  \]
  then
  \[
  |x-y|=|\eta|,\;\;d\bx d\by=2^{-2m} d\xi d\eta.
  \]
  Note that if $\bx,\by\in \supp f_R$, then  $|\bx|,|\by|\leq Rr_0/2$ and  thus
  \begin{equation}\label{suppfR}
 \bx,\by\in \supp f_R\,\Ra\, |\xi|,\;|\eta|<\frac{1}{2}\lp |\xi+\eta|+|\xi-\eta|\rp= |\bx|+|\by|\leq Rr_0.
  \end{equation}
  The natural projection $\pi:\widehat{\eX}\to\bR^m\times \bR^m$ can  given the explicit  description
  \[
  \bR^m \times S^{m-1}\times [0,\infty)\ni (\xi, \bnu,r)\mapsto (\xi,\eta)=(\xi, r\bnu)\in\bR^m\times \bR^m.
  \]
  For $R\in (R_0, \infty]$   we set
  \[
 w_R(\bx,\by)= |\bx-\by|^{m-2} \hrho_R(\bx,\by).
  \]
 Lemma \ref{lemma: key_est}(i)   implies  that  for any $C>0$
 \[
 \sup_{R\in (R_0,\infty]} \sup\Vert K_\ga^R\Vert_{C^6(RB)}<\infty.
 \]
 Since  $w_R(\bx,\by)$ depends only on $\bx-\by$ we   deduce from Proposition \ref{prop: rad_blow}   that   
  \begin{equation}\label{sup_wR}
 \sup_{R\in (R_0,\infty]} \sup_{\substack{\bx,\by\in RB\\ 0<|\bx-\by|\leq C_5}} \lv w_R(\bx,\by)\rv <\infty.
 \end{equation}
  It is easy to see that $\trho_R\circ \pi$ admits a continuous extension to the blow-up. Using (\ref{tildeKR}) and (\ref{sup_wR}) we deduce that for any $p>0$ there exists a constant $K_p>0$, independent of $R$, such that
  \begin{equation}\label{sup_deR}
  |x-y|^{m-1} \lv \Delta_R(\bx,\by)\rv \leq K_p\lp 1+|x-y|\rp^{-p+m-1},\;\;\forall \bx,\by\in RB
  \end{equation}
  Set
  \[
 \delta_R(\xi,\eta)= \Delta_R\lp \pi(\xi,\eta)\rp
 \]
Since $\Delta_R(\bx,\by)$ depends only on $\by-\bx$ we deduce that $\delta_R(\xi,\eta)$ is independent of $\xi$.   We have
  \[
 I(R)= \int_{\eX}\Delta_R(\bx,\by) f_R^{\boxtimes 2}(\bx,\by) d\bx d\by =\int_{\substack{|\bx|,|\by|\leq Rr_0/2\\\bx\neq \by}}\Delta_R(\bx,\by) f_R^{\boxtimes 2}(\bx,\by) d\bx d\by
  \]
  \[
    \stackrel{(\ref{suppfR})}{=}\frac{1}{2^{2m}}\int_{\substack{|\xi|<Rr_0,\\ |\bnu|=1,\,r\in(0,Rr_0)}}r^{m-1}\delta _R\lp  \xi,r\bnu\rp f _R\Lp  \frac{\xi+r\bnu}{2}\Rp f_R\Lp\frac{\xi-r\bnu}{2}\Rp dr\vol_{S^{m-1}}[d\bnu] d\xi
  \]
  ($\xi=2R\zeta$, $\delta_R(\xi,r\nu)=\delta_R(0,r\nu)$)
  \[
=R^m \int_{|\zeta|\leq  }\underbrace{\left(2^{-m}\int_{\substack{|\bnu|=1\\
r\in (0,Rr_0)}}r^{m-1} \delta _R\lp 0 ,r\bnu\rp f \Lp  \zeta+\frac{r\bnu}{2R}\Rp f(\zeta-\frac{r\bnu}{2R}\Rp dr\vol_{S^{m-1}}[d\bnu]\right)d\zeta}_{=:J(R)}.
 \]
  Note that
  \[
  \delta _R\lp 0 ,r\bnu\rp=\hrho_R(r\bnu/2,-r\bnu/2)-\trho_R(r\bnu/2,-r\bnu/2)
  \]
  and  for $r>0$, $|\bnu|=1$ fixed
  \[
 \lim_{R\to \infty} \delta _R\lp 0 ,r\bnu\rp=\delta_\infty(0,r\nu)= \hrho_\infty(r\bnu/2,-r\bnu/2)-\trho_\infty(r\bnu/2,-r\bnu/2).
 \]
 We deduce  from (\ref{suppfR}) and (\ref{sup_wR}) that,  for any $p>0$,   there exists $K_p>0$ such that, for any $R>R_0$, $|\zeta|<r_0/2$, $|\bnu|=1$ and $r\leq Rr_0$,  we have
 \[
 \Lv r^{m-1} \delta _R\lp 0 ,r\bnu\rp f \Lp  \zeta+\frac{r\bnu}{2R}\Rp f\Lp \zeta-\frac{r\bnu}{2R}\Rp \Rv\leq K_p\Vert f\Vert^2\lp 1+r\rp^{-p+m-1}.
 \]
  For $p>m$ we have
 \[
  \int_{|\zeta|\leq r_0/2}\left(\int_{(0,\infty)\times S^{m-1}} \lp 1+r\rp^{-p+m-1} dr \vol_{S^{m-1}}\lb d\bnu\right)d\zeta<\infty.
  \]
  The dominated  convergence theorem implies  that $J(R)$ has a finite limit  as $R\to\infty$.  More precisely
  \[
  \lim_{R\to\infty}J(R)=\int_{|\zeta|\leq r_0/2}\underbrace{ 2^{-m}\left(\int_{\substack{|\bnu|=1\\
r>0}}r^{m-1} \delta _\infty \lp 0 ,r\bnu\rp dr\vol_{S^{m-1}}[d\bnu]\right)}_{=:Z_m(\ga)} f(\zeta)^2 d\zeta.
\]
This concludes the proof of Theorem \ref{th: main}.

 \appendix
 \section{Some abstract ampleness  criteria}\label{s: a}
 We begin we an abstract result  that  will  be our main tool for detecting  ample  Gaussian fields.
 
 \begin{proposition}\label{prop: rand_series} Let $\bsX$ be  a separable Banach space with norm $\Vert-\Vert$. Let    $(x_n)_{n\geq 0}$ be a sequence in $\bsX$ and  $(c_n)_{n\geq 0}$  a sequence of positive real numbers such that
\[
\sum_{n\geq 1}c_n\Vert x_n\Vert<\infty. 
\]
Denote by $\bsY$ the closure of the span of $(x_n)_{n\geq 1}$. Let $(A_n)_{n\geq 1}$ be a sequence of independent standard normal random variables defined on the probability space $(\Omega,\eS,\bP)$. Then the following hold.

\begin{enumerate} 

\item There exists a negligible subset $\eN\in\eS$ such that the series
\[
\sum_{n\geq 1}A_n(\omega) c_nx_n
\]
converges in $\bsX$ to an element in $\bsY$  for any $\omega\in \Omega\setminus \eN$.

\item The map $S:\Omega\to \bsY$ defined by
\[
S(\omega)=\begin{cases}
\sum_{n\geq 1} A_n(\omega) c_nx_n, &\omega\in \Omega\setminus \eN,\\
0, & \omega\in \eN
\end{cases}
\]
is Borel measurable and the push-forward  $\Gamma_S:=S_\# \bP$ is a nondegenerate Gaussian measure on $\bsY$. 

\item For any  nonempty open subset $\eO\subset \bsY$, $\bP\lb S\in \eO\rb>0$.

\end{enumerate}
\end{proposition}

\begin{proof} (i) We will show that the random scalar series
\[
\sum_n |A_n|c_n\Vert x_n\Vert
\]
is $\as$ convergent. According  to Kolmogorov's two-series  theorem  this happens if the positive  random variables $X_n=|A_n|\cdot c_n\Vert x_n\Vert$ satisfy
\[
\sum_{n\geq 1}\bE\lb X_n\rb<\infty\;\;\mbox{and}\;\;\sum_{n\geq 1}\bE\lb X_n^2\rb <\infty.
\]
Now observe that
\[
\bE\lb |A_n|\rb= 2\frac{1}{\sqrt{2\pi}}\int_0^\infty x e^{-x^2/2} dx= \sqrt{\frac{2}{\pi}},
\]
\[
\sum_{n\geq 1}\bE\lb X_n\rb =\sqrt{\frac{2}{\pi}}\sum_{n\geq 1}c_n\Vert x_n\Vert<\infty
\]
and 
\[
\sum_{n\geq 1}\bE\lb X_n^2\rb =\sum_{n\geq 1} c_n^2\Vert x_n\Vert^2<\infty.
\]
(ii) Define $S_n:\Omega\to \bsY$
\[
S_n(\omega)=\begin{cases}
\sum_{k=1}^n A_k(\omega)c_kx_k, &\omega\in \Omega\setminus \eN,\\
0, & \omega\in \eN.
\end{cases}
\]
The  maps $S_n$ are measurable since the addition operation on a separable Banach space is a measurable map. The map $S$ is measurable since for any $\xi\in \bsY^*$ the function  $ \lan\xi, S\ran$ is measurable as limit of the measurable functions  $\lan \xi, S_n\ran$.

To see that $\Gamma_S$ is a Gaussian map  let $\xi\in \bsY^*$. Then
\[
\lan \xi, S(\omega)\ran=\lim_{n\to \infty} \lan \xi, S_n(\omega)\ran.
\]
The random variables
\[
\lan \xi, S_n\ran=\sum_{k=1}^n A_n c_n\lan \xi, x_n\ran
\]
are Gaussian as sum of independent Gaussians.  Since the limit of Gaussian random variables is also Gaussian   we deduce  that $\lan\xi, S\ran$ is Gaussian with variance
\[
v[\xi]=\sum_{n\geq 1}c_n^2\lv \lan \xi, x_n\ran\rv^2.
\]
Since $(x_n)$ spans a dense subspace of $\bsY$, we deduce that for any $\xi\in \bsY^*\setminus 0$    there exists $n$ such that $\lan \xi, x_n\ran\neq 0$. This proves that $\Gamma_S$ is nondegenerate.  Part (iii) now follows from the Support Theorem   \cite[Thm. 3.6.1]{Bog} or \cite[Cor. 4.2.1]{Stroo}.
\end{proof}

\begin{proposition}\label{prop: approx_ample} Suppose that $\bsU$ is a Banach space with norm $\Vert-\Vert$, $\bsT$ is a compact metric space, $N\in \bN$ and 
\[
G: \bsU^N\times \bsT\to [0,\infty) ,\;\;(u_1,\dotsc, u_N,t)\mapsto G(u_1,\dotsc, u_N,t) \in [0,\infty)
\]
is a continuous  function. We  define
\[
G_* : \bsU^N\to [0,\infty),\;\; G_*(u_1,\dotsc, u_N):= \min_{t\in \bsT} G(u_1,\dotsc, u_N, t).
\]
Suppose that there exist $v_1,\dotsc, v_N\in \bsU$ such that  $G_*(v_1,\dotsc, v_N)=r_0>0$.  Then,  for any $r\in (0,r_0)$, there exists $\ve=\ve(r)>0$ such  that
\[
\forall u_1,\dotsc, u_N\in \bsU,\;\;\forall  i=1,\dotsc N,\;\;\Vert u_i-v_i\Vert<\ve\Ra G_*(u_1,\dotsc, u_N)>r.
\] 
In particular if
\[
U_1\subset U_2\subset\cdots 
\]
is an increasing  sequence  of  finite dimensional subspaces of  $\bsU$ whose union is a dense subspace of $\bsU$, then there exists $\nu\in \bN$ and
\[
u_{1,\nu},\dotsc, u_{N,\nu}\in U_\nu
\]
such that  $G_*\lp u_{1,\nu},\dotsc, u_{N,\nu}\rp>0$.

\end{proposition}

\begin{proof} We argue by contradiction  suppose there exists $r_1\in (0,r_0)$ and sequences  in $\bsU$
\[
\lp u_{i,\nu}\rp_{\nu\in \bN},\;\;i=1,\dotsc, N,
\]
such that
\[
\lim_{\nu\to\infty}\Vert u_{i,\nu}-v_i\Vert=0, \;\;\forall i=1,\dotsc, N,
\]
and 
\[
G_*(u_{1,\nu},\dotsc, u_{N,\nu})\leq r_1,\;\; \forall \nu. 
\]
Next choose $t_\nu\in \bsT$ such that
\[
G\lp u_{1,\nu},\dotsc, u_{N,\nu}, t_\nu\rp= G_*\lp u_{1,\nu},\dotsc, u_{N,\nu}\rp.
\]
Upon extracting a subsequence we can assume that $t_\nu$ converges in $\bsT$ to some point $t_\infty$. Then
\[
r_1\geq \liminf _{\nu\to\infty } G_*\lp u_{1,\nu},\dotsc, u_{N,\nu}\rp=\liminf_{\nu\to \infty } G\lp u_{1,\nu},\dotsc, u_{N,\nu}, t_\nu\rp
\]
\[
= G\lp v_1,\dotsc, v_N, t_\infty)\geq r_0>r_1.
\]
\end{proof}

With $\bsT$ a compact metric space as above.  Let $E\to\bsT$ be a rank $r$  topological real vector bundle over $\bsT$ equipped with a continuous  metric $h$.  We will refer to the pair $(E,h)$ as a metric vector bundle. For $t\in \bsT$ we denote by $|-|_t$ the norm on the fiber $E_t$ induced by $h$.  The space $C^0(E)$   of continuous sections $E$ is a Banach space with respect to the norm
\[
\Vert u\Vert:=\sup_{t\in\bsT}\lv u(t)\rv_t,\;\;u\in C(E).
\]
\begin{definition}\label{def: ample_Ban} An \emph{ample Banach space of sections} of $E$  is  a Banach space   $\bsU\subset   C^0(E)$   continuously embedded in $C^0(E)$  such that  \index{Banach space! ample}
\[
\forall t \in \bsT,\;\;\spa\big\{\,u(t),\;\;u\in \bsU\,\big\}= E_t.
\]
Let $k\in \bN$. We say that the Banach space $\bsU$ is \emph{$k$-ample} \index{Banach space! $k$-ample}  if for any \emph{distinct } points $t_1,\dotsc, t_k\in \bsT$ the map
\[
\bsU\ni u\mapsto  u(t_1)\oplus \cdots \oplus u(t_k)\in E_{t_1}\oplus \cdots \oplus E_{t_k}
\]
is onto.  \qed
\end{definition}

\begin{ex} The space $C^0(E)$ is a $k$-ample Banach space of continuous sections  of $E\to\bsT$ for any $k\in\bN$.  If $\bsT$ is a compact smooth manifold and $E\to\bsT$ is a smooth vector bundle, then  each of the spaces $C^\ell (E)$, $\ell\in\bN$, is a $k$-ample Banach space of sections of $E$ for any $k\in \bN$.\qed
\end{ex}

\begin{corollary}\label{cor: filter_ample}  Let $ E\to\bsT$  be a real metric vector bundle over the compact metric space $\bsT$. Suppose that $\bsU\subset C^0(E)$ is an ample Banach space  of sections  of $E$ and
\[
U_1\subset U_2 \cdots 
\]
is  an increasing sequence of finite dimensional subspaces of $\bsU$ such that
\[
U_\infty=\bigcup_{\nu\in \bN} U_\nu
\]
is dense in $\bsU$. Then there exists $\nu\in \bN$, for any $t\in \bsT$, the evaluation map
\[
\ev_t:U_\nu\to E_t\;\;\mbox{is onto}.
\]
\end{corollary}

\begin{proof}   Using the compactness of $\bsT$ and the openess  of the surjectivity condition  we can find $v_1,\dotsc, v_N\in\bsU$ such that
\[
\forall t \in \bsT,\;\;\spa\big\{\, v_1(t),\dotsc, v_N(t)\,\big\}= E_t.
\] 
For every $u_1,\dotsc, u_N\in U$ and $t\in\bsT$ define 
\[
S_{u_1,\dotsc, u_N, t}:\bR^N\to E_t,\;\; S_{u_1,\dotsc, u_N,t}(\bx)=\sum_{k=1}^N x_k u_k(t)
\]
and 
\[
G(u_1,\dotsc, u_N,t)=\det\lp  S_{u_1,\dotsc, u_N,t} S_{u_1,\dotsc, u_N,t}^*\rp\geq 0.
\]
Note that 
\[
\spa \big\{\,u_1(t), \dotsc, u_N(t)\,\big\}=E_t\,\Llra \, G(u_1,\dotsc, u_N,t)>0.
\]
 The resulting map $G: \bsU^N\times \bsT\to [0,\infty)$ is continuous and
\[
G(v_1,\dotsc, v_N,t)>0,\;\;\forall t\in \bsT
\]
 Hence
 \[
 G_*(v_1,\dotsc, v_N)=\inf_{t\in\bsT}G(v_1,\dotsc, v_N,t)>0.
 \]
Using Proposition  \ref{prop: approx_ample}, we deduce that there exists $\nu\in \bN$ and $u_{1,\nu},\dotsc, u_{N,\nu}\in U_\nu$ such that
\[
G_*(u_{1,\nu},\dotsc, u_{N,\nu})>0.
\]
 Hence
 \[
 \ev_t:\spa\big\{ \, u_1,\dotsc, u_N\,\big\}\subset \bsU\to E_t\;\mbox{is onto},\;\;\forall t\in \bsT.
 \]
 A fortiori, this implies  that
  \[
  \ev_t:U_\nu\to E_t\;\mbox{is onto},\;\;\forall t\in \bsT.
  \]
\end{proof}

\begin{corollary}\label{cor: filter_2ample}  Let $ E\to\bsT$  be a real metric vector bundle over the compact metric space $\bsT$. Suppose that $\bsU\subset C^0(E)$ is a $2$-ample Banach space  of sections  and
\[
U_1\subset U_2 \cdots 
\]
is  an increasing sequence of finite dimensional subspaces of $\bsU$ such that
\[
U_\infty=\bigcup_{\nu\in \bN} U_\nu
\]
is dense in $\bsU$. Then, for any open neighborhood $\eO$ of the diagonal $\Delta\subset \bT\times \bT$,  there exists $\nu\in \bN$ such that  for any $(t_1,t_2)\in \bsT^2\setminus \eO$, the evaluation map
\[
\ev_{t_1,t_2}:U_\nu\to E_{t_1}\oplus E_{t_2}\;\;\mbox{is onto}.
\]
\end{corollary}

\begin{proof}   For $\ut\in \bsT^2$ and $u\in \bsU$ we set
\[
u(\ut):=u(t_1)\oplus u(t_2),\;\;E_{\ut}=E_{t_1}\oplus E_{t_2},\;\;\ev_\ut(u)=u(\ut).
\]
Using the compactness of $\bsT^2\setminus \eO$ and the openness  of the surjectivity condition  we can find $v_1,\dotsc, v_N\in\bsU$ such that
\[
\forall \ut \in \bsT^2\setminus \eO,\;\;\spa\big\{\, v_1(\ut),\dotsc, v_N(\ut)\,\big\}= E_{\ut}.
\] 
For every $u_1,\dotsc, u_N\in \bsU$ and $\ut\in\bsT^2$ define 
\[
S_{u_1,\dotsc, u_N, \ut}:\bR^N\to E_\ut,\;\; S_{u_1,\dotsc, u_N,\ut}(\bx)=\sum_{k=1}^N x_k u_k(\ut)
\]
and 
\[
G(u_1,\dotsc, u_N,\ut)=\det\lp  S_{u_1,\dotsc, u_N,\ut} S_{u_1,\dotsc, u_N,\ut}^*\rp\geq 0.
\]
Note that 
\[
\spa \big\{\,u_1(\ut), \dotsc, u_N(\ut)\,\big\}=E_\ut\,\Llra \, G(u_1,\dotsc, u_N,\ut)>0.
\]
Thus
\[
G(u_1,\dotsc, u_N,\ut)>0\,\;\Llra\; \ev_t:\spa\big\{ \, u_1,\dotsc, u_N\,\big\}\subset \bsU\to E_\ut\;\mbox{is onto}.
\]
 The resulting map $G: \bsU^N\times \lp \bsT^2\setminus \eO\rp\to [0,\infty)$ is continuous and 
 \[
 G_*(v_1,\dotsc, v_N)=\inf_{\ut\in\bsT^2\setminus \eO} G(v_1,\dotsc, v_N,\ut)>0.
 \]
Using Proposition  \ref{prop: approx_ample}, we deduce that there exists $\nu\in \bN$ and $u_{1,\nu},\dotsc, u_{N,\nu}\in U_\nu$ such that
\[
G_*(u_{1,\nu},\dotsc, u_{N,\nu})>0.
\]
 Hence
 \[
 \ev_\ut:\spa\big\{ \, u_1,\dotsc, u_N\,\big\}\subset \bsU\to E_\ut\;\mbox{is onto},\;\;\forall \ut\in \bsT^2\setminus \eO.
 \]
 A fortiori, this implies  that
  \[
  \ev_\ut:U_\nu\to E_\ut\;\mbox{is onto},\;\;\forall \ut\in \bsT^2\setminus \eO.
  \]
\end{proof}

\begin{proposition}\label{prop: rand_series_ample}  Suppose that $E\to\bsT$ is a topological  metric vector bundle over the compact metric space $\bsT$. Let  $\bsX\subset C^0(E)$  be an ample  Banach space  of sections of $E$ embedded continuously  in $C^0(T)$. 

Suppose that $(u_n)_{n\in \bN}$ is a sequence of sections  in $\bsX$ such that $\spa\big\{ u_n,\;\;n\in \bN\,\big\}$ is dense in $\bsX$ and exists $\alpha>0$  such that
\begin{equation}\label{poly_norm}
\Vert u_n\Vert_{\bsU}=O(n^\alpha)\;\;\mbox{as $n\to\infty$}.
\end{equation}
Fix a sequence of positive real numbers $(\lambda_n)_{n\geq 0}$ such that 
\begin{equation}\label{poly_eval}
\liminf_{n\to \infty}\frac{\lambda_n}{n^\beta}>0, 
\end{equation}
for some  $\beta>0$.  Let $\ga\in \eS(\bR)$ be an even Schwartz  function such that $\ga(0)=1$. Fix a sequence of  $\iid$ standard normal random variables $(X_n)_{n\geq 0}$. Then the following hold.

\begin{enumerate}

\item For any $R >0$ the random series  
\begin{equation}\label{rand_series_ample_bun}
\sum_{n\in \bN} \ga\lp  \lambda_n/R\rp X_n u_n
\end{equation}
converges  $\as$ in $\bsX$. Denote by $\Phi^R$ the resulting continuous  Gaussian section  of $E$.

\item There exists $R_0$ such that $\forall R>R_0$ the  Gaussian section $\Phi^R$ is ample.

\end{enumerate}
\end{proposition}

\begin{proof} (i)   Since $\ga$ is a Schwartz function we deduce from (\ref{poly_norm}) and  (\ref{poly_eval}) that
\[
\sum_{n\to\infty}\lv \ga\lp \lambda_n/R\rp\rv\Vert u_n\Vert_{\bsX} <\infty,\;\;\forall \hbar >0
\]
The convergence of the random series (\ref{rand_series_ample_bun}) in $C^0(E)$  follows from  Proposition \ref{prop: rand_series}. 

\smallskip

\noindent (ii) For $\hbar>0$  we set
\[
\eN_R:=\big\{ n\in \bN;\;\; \ga\lp\lambda_n/R\rp\neq 0\,\big\}
\]
and denote by $\bsY^R$ the closure in $\bsX$ of
\[
\spa\big\{\, u_n;\;n\in \eN_R\,\big\}.
\]
According to Proposition \ref{prop: rand_series} the above random series  defines a \emph{nondegenerate} Gaussian $\bGamma^\hbar$  measure on the Banach space $\bsY^\hbar$.

Set
\[
U_\nu:=\spa\big\{\, u_1,\dotsc, u_\nu\,\big\}.
\]
Since $\ga(0)=1$, we deduce that
\[
\exists r_0>0,\;\;  \forall |t|\leq r_0,\;\; \lv \ga(t)\rv\geq1/2.
\]
Hence, for any $\nu\in \bN$ there exists $R=R(\nu)>0$ so that
\[
\forall R\geq R(\nu),\;\;\max_{1\leq k\leq \nu}\lambda_k/R<r_0,
\]
i.e.,
\[
U_\nu\subset \bsY^R,\;\;\forall  R\geq R(\nu).
\]
Corollary \ref{cor: filter_ample} implies that there exists $\nu_0\in \bN$ such that
\[
\forall t\in \bsT,\;\;\ev_t:U_{\nu_0}\to E_t\;\;\mbox{is onto}.
\]
Set $\hbar_0=\hbar(\nu_0)$ such that $U_{\nu_0}\subset \bsY^R$, $\forall R\leq R_0$.

We will show  that for  any $t\in \bsT$ and any $R\geq R_0$, the Gaussian vector $\Phi^R(t)$ is nondegenerate, i.e., for  any open set $\eO\subset E_t$, $\bP\lb \Phi^R(t)\in \eO\rb >0$. Equivalently, this means
\[
\bGamma^R\lb \ev^{-1}_t(\eO)\rb >0.
\]
Since $\bGamma^R$ is a nondegenerate  Gaussian measure on $\bsY^R$, it suffice to show that the open subset $ \ev^{-1}_t(\eO)\subset \bsY^R$ is nonempty. This is indeed the case  since  $\ev^{-1}_t(\eO)\supset \ev^{-1}_t(\eO)\cap U_{\nu_0}\neq \emptyset$.
\end{proof}

Suppose that  $M$ is a smooth compact manifold. Denote by $C^k(E)$ the vector space of sections of $E$ that are $k$-times continuously differentiable.  We  need to define on $C^k(E)$ a structure of separable Banach space and to do so  we need to make some choices.   
 
 \begin{itemize}
 
 \item Fix a smooth Riemannian  metric $g$ on $M$.
 
 \item Fix a smooth $h$ metric on $E$. We denote by $(-,-)_{E_x}$  the induced inner product on $E_x$.
 
 \item Fix a connection (covariant derivative)  $\nabla^h$ on $E$ that is compatible with the metric  $h$. \index{connection}\index{covariant derivative}
 
 \end{itemize}
 
 We will refer to such choices as \emph{standard choices}.  \index{standard choices} There are several  geometric objects canonically induced  by these choices; see \cite[Sec. 3.3]{Nico_geom}.
 
  First, the metric $g$ determines a  a Borel measure $\vol_g$ on $M$,  classically referred to as  the  \emph{volume element} or the  \emph{volume density}.  Next, the metric determines  the \emph{Levi-Civita connection}\index{connection! Levi-Civita} \index{Levi-Civita connection}  $\nabla^g$  on $TM$. The metric $g$  also  determines metrics on all the tensor bundles $TM^{\otimes p}\otimes (T^*M)^{\otimes q}$ and the connection $\nabla ^g$ determines connections on these bundles compatible with the metrics induced by $g$.   To ease the notational burden we will denote by $\nabla^g$  each of these connections. 
  
  Similarly, the metric $h$ induces  metrics  in all the bundles $E^{\otimes p}\otimes (E^*)^{\otimes q}$ and the connection $\nabla^h$ determines connections   on these bundles  compatible with the induced metrics. We will denote by $\lv -\rv_x$ the Euclidean norms in any of the spaces $(T_x^*M)^{\otimes q} \otimes E^{\otimes p}$.  We  define the \emph{jet bundle}
  \begin{equation}\label{jet_bundle}
  J_k(E):= \bigoplus_{j-0}^k T^*M^{\otimes j}\otimes E.
  \end{equation}
  The connections $\nabla^g$ and $\nabla^h$ induce a connection $\nabla=\nabla^{g,h}$ on the bundle $(T^* M)^{\otimes k}\otimes  E$
  \[
  \nabla : C^1\lp (T^*M)^{\otimes k}\otimes E\rp\to C^0\lp (T^*M)^{\otimes k+1}\otimes E\rp.
  \]
 We denote by  $\nabla^q$ the composition
  \[
  C^q( E)\stackrel{\nabla}{\ra}  C^{m-1}\lp T^*M \otimes E\rp \stackrel{\nabla}{\ra}\cdots  \stackrel{\nabla}{\ra} C^1\lp (T^*M)^{\otimes q-1}\otimes E\rp\stackrel{\nabla}{\to} C^0\lp (T^*M)^{\otimes q}\otimes E\rp.
  \]
  For  every section $\psi\in C^k(E)$ we set
  \[
  J_k(\psi)=J_k(\psi, \nabla)=\bigoplus_{k=0}^k \nabla^k \psi
  \]
  \[
  \Vert u\Vert_{C^k}=\sum_{j=0}^q\Vert \nabla^j \psi\Vert,
  \]
  where 
  \[
  \Vert \nabla^j u\Vert=\sup_{x\in M} \lv \nabla^j u(x)\rv_x.
  \]
    Note that $J_k(\psi)$ is a continuous section of $J_k(E)$.The resulting  normed spaces is a separable  Banach space.   The norm  $\Vert-\Vert_{C^k}$ depends on the standard choices, but different standard choices  yield equivalent norms. Fix one such norm and denote by $C^k(E)$  the  resulting separable Banach space.

 \begin{corollary}\label{cor: rand_series_ample} Suppose that $E\to M$ is a smooth real vector bundle over the compact smooth manifold $M$. Fix a smooth  Riemann metric $g$ on $M$,  a smooth metric $h$ on $E$ and a smooth connection  on $E$ compatible with $h$. Let $k\in \bN$ and suppose that $(\phi_n)_{n\in \bN}$ is a sequence  of $C^k$ sections of $E$ that span  a dense subset of $C^k(E)$. Suppose that
 \begin{equation}\label{poly_norm_1}
 \Vert \phi_n\Vert_{C^k(E)}= O(n^\alpha)\;\;\mbox{as $n\to\infty$},
 \end{equation}
 for some $\alpha>0$.  Fix a sequence of positive numbers $(\lambda_n)_{n\in\bN}$ satisfying (\ref{poly_eval}). Let $(X_n)_{n\in \bN}$ be a sequence of $\iid$ standard normal random variables and suppose that $\ga\in \eS(\bR)$ is an even Schwartz function such that $\ga(0)=1$. Then the following hold.
 
 \begin{enumerate}

\item For any $R >0$ the random series  
\begin{equation}\label{rand_series_ample_1}
\sum_{n\in \bN} \ga\lp  \lambda_n/R\rp X_n \phi_n
\end{equation}
converges  $\as$ in $C^k(E)$. Denote by $\Phi^R$ the resulting $C^k$  Gaussian section  of $E$.

\item There exists $R_0>$ such that $\forall R>R_0$ the  Gaussian section $\Phi^R$ is $J_k$-ample, i.e., for any $\bx\in M$ the Gaussian vector
\[
J_k\Phi^R(\bx)=\bigoplus_{j=0}^k \nabla^j\Phi^R(\bx)
\]
is nondegenerate.

\end{enumerate}

 \end{corollary}
 
 \begin{proof} (i)  This follows from Proposition \ref{prop: rand_series_ample}.
 
 \smallskip
 
 \noindent (ii) Consider the jet bundle $J_k(E)\to M$; see (\ref{jet_bundle})\index{jet bundle}  We have a continuous linear
 \[
 C^k(E)\to C^0\lp J^k(E)\rp,\;\;\phi \mapsto J_k(\phi).
 \]
 Denote by $\bsU$ the image of this map.  It is a closed\footnote{Here we are using the classical  fact that if a   sequence of $C^1$-function $(u_n)$ has the property that both $(u_n)$ and  their differentials $(d u_n)$ converge uniformly to $u$ and respectively $v$, then $u$ is $C^1$ and $du=v$.}  subspace of  $C^0\lp J^k(E)\rp$.  Then the random series 
 \[
 \sum_{n\in \bN} \ga\lp  \lambda_n/R\rp X_n J_k(\phi_n)
 \]
 converges $\as$ uniformly to $J_k(\Phi^R)$.  Now observe that $\bsU$ is an ample Banach space of sections of $J^k(E)$. Indeed, using smooth partitions of unity  we can find   $\psi_1,\dotsc , \psi_N\in C^k(E)$ such that, for any $x\in M$,  
 \[
 \spa\big\{ \, J_k(\psi_1(x),\dotsc, J_k(\psi_N)(x)\,\big\}=J_k(E)_x.
 \]
 Proposition  \ref{prop: rand_series_ample} now implies that $J_k(\Phi^R)$ is an ample Gaussian section of $J_k(E)$.
 \end{proof}
 
 \begin{remark} In applications $(\phi_n)$ are eigenfunctions of the Laplacian $\Delta$ on a Riemann  manifold $(M,g)$ and $\Delta\phi_n=\lambda_n^2\phi_n$.  The covariance kernel of $\Phi^R$ is then the Schwartz kernel of the smoothing operator $\ga\lp \hbar\sqrt{\Delta}\rp^2$, $\hbar=R^{-1}$. If $\ga(x)=e^{-x^2/2}$ $\hbar=t^{1/2}$, then $\ga\lp \hbar\sqrt{\Delta}\rp^2= e^{-t\Delta}$, the heat operator on $M$. \qed
 \end{remark}
 
 \begin{corollary}\label{cor: hbar_Jk} Fix an even Schwartz function $\ga\in \eS(\bR)$ and consider  the random  Fourier series $F^R_\ga$   defined in (\ref{rand_fourier0}). We regard it as a random  smooth  function on the torus  $\bT^m$.  Then for any $k\in \bR$ there exists $R=R_k>0$ such that, for any $R> R_k$ the function $F^R_\ga$ is $J_k$-ample.  In particular it is $\as$ Morse for $R>R_1$. \qed
 
 \end{corollary}
 
  \begin{lemma}\label{lemma: rand_series_ample} Suppose that $E\to M$ is a smooth real vector bundle over the compact smooth manifold $M$. Fix a smooth  Riemann metric $g$ on $M$,  a smooth metric $h$ on $E$ and a smooth connection  on $E$ compatible with $h$. Let $k\in \bN$ and suppose that $(\phi_n)_{n\in \bN}$ is a sequence  of $C^k$ sections of $E$ that span  a dense subset of $C^k(E)$. Set
 \[
 U_\nu:=\spa\lbr \phi_1,\dotsc,\phi_\nu\rbr
 \]
 Then there   exists  $\nu_0>0$ such that  $\forall \nu\geq \nu_0$ the following hold.
 
 \begin{enumerate}
 
 \item  For any $t\in M$ and any $\nu \geq \nu_0$ the map
 \[
 U_\nu\ni u\mapsto J_1(u)_t\in J_1(E_t)
 \]
 is onto. Above, $J_1(u)_t$ is the $1$-jet of $u$ at $t$, $J_1(u)_t=u(t)\oplus \nabla u (t)\in E_t\oplus T^*_tM\otimes E_t$.
 
 \item  For any $\ut \in M^2\setminus \Delta$   the map
 \[
 U_\nu\ni u\mapsto u(\ut)\in E_\ut
 \]
 is onto.
 \end{enumerate}
 
 \end{lemma}
 
 \begin{proof}  The  space $C^k(E)$ is $J_1$-ample  and arguing as in the proof of Corollary \ref{cor: filter_ample}  we deduce that  there exists $\nu_1\in \bN$  such  that  for any $\nu\geq \nu_1$ and  $t\in M$  the map
 \[
 U_{\nu}\ni u\mapsto J_1(u)_t\in J_1(E)_t
 \]
 is ample.
 
 The argument  at the beginning of \cite[Sec. 3.3]{GS23} based on Kergin interpolation  shows  that there exists an open neighborhood $\eO$ of the diagonal $\Delta\in M^2$ such that  $\forall \nu\geq \nu_1$ and any $\ut \in \eO\setminus \Delta$  the map
 \[
 U_\nu\ni u\mapsto u(\ut)\in E_\ut 
 \]
 is onto.
 
 Corollary \ref{cor: filter_2ample} implies that there exists $\nu_0>0$ such that $\forall \nu\geq \nu_2$ and any $\ut\in M^2\setminus \eO$  the map
 \[
 U_\nu\ni u\mapsto u(\ut)\in E_\ut 
 \]
 is onto. Then $\nu_0=\max(\nu_1,\nu_2)$  has all the claimed properties.

 \end{proof}
 
  \begin{corollary}\label{cor: hbar_2amp}Fix an even Schwartz function $\ga\in \eS(\bR)$ and consider  the random  Fourier series $F^R_\ga$   defined in (\ref{rand_fourier0}). We regard it as a random  smooth  function on the torus  $\bT^m$. Then there exists $R=R_{2,2}>0$ such that, for any $R>R_{2,2}$ the function $F^R_\ga$ is $J_2$-ample and $\nabla F^R_\ga$ is $2$-ample.  \qed
 \end{corollary}
 
 \section{Variance estimates}\label{s: b}
 
  It has been known for some time  that under certain conditions the number of zeros  in a box of a Gaussian field $F$  has finite variance, \cite{AnLe23, BMM22,  EL, GS23}.  In this warm-up subsection  we  use the ideas in the above references to obtain such estimates  for the variance in terms of the  covariance kernel.  Here an in the sequel   
 
   Suppose that $\bsU$ and $\bsV$   are finite dimensional  real Euclidean spaces of the same dimension $m$  and $\mV\subset \bsV$ is an open set.    If $f:\mV\to \bsU$ is a $C^k$-map, we denote  by $f^{(k)}(v)$ its $k$-th differential at $v\in \mV$. We view $f^{(k)}(v)$  as an element of $\Sym^k(\bsV,\bsU)$ the space of symmetric $k$-linear maps $\bsV^k\to\bsU$.
   
   Let $F:\mV\to \bsU$ be  a Gaussian random field whose covariance kernel $\eK_F$ is $C^6$. In particular, this implies that $F$    is $\as$  $C^2$.  
   
   For any box $B\subset \mV$  we denote by $Z_B$ the number of zeros  of $F$ in $B$, i.e., $Z_B=Z[B, F]$. Let  $\mV^2_*:=\mV^2\setminus \Delta$, where $\Delta$ is the diagonal
\[
\Delta:=\big\{\;(v_0,v_1)\in \mV^2;\;\; v_0= v_1\,\big\}.
\]
Define $B^2_*$ in a similar fashion.    Consider the random field
\[
\widehat{F}=: \mV^2_*\to\bsU\oplus \bsU,\;\;\hat{F}(v_0,v_1)= F(v_0)\oplus F(v_1).
\]
Note that
\[
Z[\widehat{F}, B^2_*]=Z_B\lp Z_B-1\rp.
\]
\emph{Suppose that $F|_B$ is $2$-ample},  i.e., for any $\uv=(v_0,v_1)\in B^2_*$ the Gaussian vector $F(v_0)\oplus F(v_1)$ is nondegenerate.   We deduce  from  the local Kac-Rice formula  (\ref{KR}) that  $E\lb Z_B\rb<\infty$, and
 \[
 \bE\lb Z_B\lp Z_B-1\rp\rb=\int_{B^2_*}\rho^{(2)}_G(v_0,v_1) dv_0dv_1,
 \]
 where  $\rho^{(2)}_F$ is the Kac-Rice density
 \begin{equation}\label{KR_dens}
 \rho^{(2)}_F(v_0,v_1) := \bE\lb\vert\det F'(v_0)\det F'(v_1)\vert\;\lv  F(v_0)=F(v_1)=0\rb p_{\widehat{F}(v_0,v_1)}(0) .
 \end{equation}
 Note that
\[
p_{\widehat{F}(v_0,v_1)}(0)= \frac{1}{\sqrt{\det \lp 2\pi\var[F(v_0)\oplus F(v_1)]\rp}},
\]
so $p_{\widehat{F}(v_0,v_1)}(0)$ explodes as $(v_0,v_1)$ approaches the diagonal since $F(v)\oplus F(v)$ is  degenerate for any $v\in\mV$. Thus the function  $\rho^{(2)}_F(v_0,v_1)$ might have a non integrable singularity along the diagonal  so  $E\lb Z_B^2\rb$ could be infinite.

We want to show that this is not the case and a bit more.  We will use the gauge-change trick  outlined in the introduction to his section.

\begin{proposition}\label{prop: rad_blow} Fix a box $B\subset \mV$ and $r<\dist(B,\mV)$.  Denote by $S=S(r,B)$ the compact set set
\[
S=\lbr v\in \mV;\;\;\dist(v,B)\leq r\,\rbr.
\]
Suppose that $F\rv_B$ is $C^2$,  $2$-ample and $J_1$-ample, i.e., for any $\bv\in B$ the Gaussian vector $\lp F(v),F'(v)\rp$   is nondegenerate.  Define 
\[
w_F:B^2_*\to \bR, \,\, w_F(\bx,\by)=|\bx-\by|^{m-2}\rho_F^{(2)}(\bx,\by).
\]
There exists a constant $C(m, \vol[B], r)>0$,  that depends only on $m$, $\vol[B]$  and $r$ such that
\begin{equation}\label{sup_blow}
\sup_{\bp\in B^2_*}|{w}_F(\bp)\rv \leq  C(m,\vol[B], r)\Vert \eK_F\Vert^{3m-1/2}_{C^6(S\times S)}.
\end{equation}
In particular $\var\lb Z_B\rb<\infty$.
\end{proposition}

 \begin{proof}  Our approach is a modification of the arguments in \cite[Sec. 4.2]{BMM22}. For any $v_0,v_1\in B$, $v_0\neq v_1$, the Gaussian  vector $\hat{F}(v_0,v_1)=F(v_0)\oplus F(v_1)$  is   nondegenerate.  We denote by $p_{F(v_0),F(v_1)}$ the probability density of $\hat{F}(v_0,v_1)$.

We set
\[
r(\uv):=\Vert v_1-v_0\Vert, \;\;\Xi(\uv):=\frac{1}{r(\uv)}\lp F(v_1)-F(v_0)\rp.
\]
Note that
\[
\hat{F}(\uv)=0\,\Llra\, F(v_0)=\Xi(\uv)=0.
\]
Denote by $A(\uv)$ the linear map $\bsU^2\to \bsU^2$ given by
\begin{equation}\label{Auv}
A(\uv)\left[
\begin{array}{c}
u_0\\
u_1\\
\end{array}
\right]=\left[
\begin{array}{c}
u_0\\
u_0+r(\uv)u_1\\
\end{array}
\right]=\left[
\begin{array}{cc}\one_{\bsU}& 0\\
\one_{\bsU} & r(\uv)\one_{\bsU}
\end{array}\right]\cdot \left[
\begin{array}{c}
u_0\\
u_1\\
\end{array}
\right].
\end{equation}
Thus
\[
\left[
\begin{array}{c}
F(v_0)\\
F(v_1)\\
\end{array}
\right]= A(\uv)\left[\begin{array}{c}
F(v_0)\\
\Xi(\uv)\\
\end{array}
\right].
\]
The gauge transformation  $A(\uv)$ desingularizes $\hat{F}$. Denote by $Z(\uv)$ the Gaussian vector  $(F(v_0), \Xi(\uv))$ .

The Gaussian regression formula implies that
\[
\begin{split}
\bE\lb\vert\det F'(v_0)\det F'(v_1)\vert\;\lv  F(v_0)=F(v_1)=0\rb\\
=\bE\lb\vert\det F'(v_0)\det F'(v_1)\vert\;\lv  Z(\uv)=0\rb.
\end{split}
\]
Note that
\[
p_{F(v_0),F(v_1)}=\frac{1}{\sqrt{\det\lp 2\pi \var[F(v_0)\oplus F(v_1)]\rp}}
\]
\[
= \frac{1}{|\det A|\sqrt{\det\lp 2\pi \var[F(v_0)\oplus \Xi(\uv))])}}=r(\uv)^{-m}p_{F(v_0)\oplus \Xi(\uv)}(0).
\]
We deduce that for any $\uu\in B^2_*$ we have
\begin{equation}\label{rhof}
\rho^{(2)}_F(\uv) :=r(\uv)^{-m} \bE\lb\vert\det F'(v_0)\det F'(v_1)\vert\;\lv Z(\uv)=0\rb p_{F(v_0)\oplus \Xi(\uv)}(0).
\end{equation}

\begin{lemma}\label{lemma: asymp} There exists a constant $C=C(m,\vol[B],r)>0$   depending only on $m$ and $\vol[B]$  and $r<\dist(B,\pa \mV)$ such that,  for $i=0,1$,  and any $\uv\in B^2_*$  
\[
\lv \bE\lb\vert\det F(v_i)\vert^2\lv  Z(\uv)=0\rb\rv \leq C(m,\vol[B],r)\Vert \eK_F\Vert^{m+2}_{C^6(S\times S)}  r(\uv)^2.
\]
\end{lemma}

\begin{proof}   It suffices to consider only the case $i=0$ since 
\[
F(v_0)=\Xi(\uv)=0\,\Llra \,F(v_1)=\Xi(\uv)=0.
\]
Set 
\[
\nu=\nu(\uv):=\frac{1}{r(\uv)}\lp v_1-v_0\rp,\;\;Z=Z(\uv).
\]
Let  $f(t)= F(v_0+t\nu)$. Since $F(v)$  is $\as$  $C^2$  we deduce from the first order Taylor approximation with integral remainder that
 \[
 F(v_1)-F(v_0) =f(r)-f(0)=rf'(0)+\int_0^{r} f''(t)(r-t) dt=\pa_{\nu}F(v_0)+\underbrace{\int_0^{r} f''(t)(r-t) dt}_{=:W}.
 \]
 Hence
 \[
 r\pa_{\nu}F(v_0)= F(v_0)-F(v_1)-W
 \]
 Hence, for any $p\geq 1$ we have 
 \[
 \bE\lb| r\pa_{\nu}F(v_0)|^p\, \lv Z=0\rb=\bE\lb  |F(v_0)-F(v_1)-W|^p \lv Z=0\rb =\bE\lb |W|^p \,\lv Z=0\rb.
 \]
The random variable $W$ is a centered  $\bsU$-valued Gaussian vector.  We deduce that for any $p\geq 1$ we have 
\[
\lv \bE\lb |\pa_{\nu}F(v_0)|^p \lv Z=0\rb\rv = \frac{1}{r^p} \bE\lb | W|^p\, \lv Z=0\rb^p.
\]
Note that 
\[
|W|\leq \int_0^r \Vert f''(t)\Vert_{\bsU} (r-t) dt \leq \frac{r^2}{2} \Vert F\Vert_{C^2(B)}.
\]
We deduce  that
\[
\lV \Var\lb W\rb \rV_\op \leq \frac{r^4}{4} \bE\lb   \Vert F\Vert^2_{C^2(B)}\rb.
\]
Using Corollary \ref{cor: sup_gauss_int} we deduce that 
\[
 \bE\lb | W|^p \, \lv Z=0\rb\leq C(m,p)  r^{2p}  \bE\lb \Vert F\Vert^2_{C^2(B)}\rb^{p/2},
\]
 where $C(m,p)$ is a universal constant that depends only on the dimension $m$ and on $p$. We will continue to denote by  the same symbol $C(m,p)$ various positive constants that depend only on $m$ and $p$. We deduce 
 \begin{equation}\label{asymp}
   \bE\lb \lv\pa_{\nu}F(v_0)\rb^p\lv Z=0\rb \leq C(m,p) r^p  \bE\lb \Vert F\Vert^2_{C^2(B)}\rb^{p/2}.
 \end{equation}
Extend $\nu$ to an orthonormal basis $\{\nu=\be_1,\be_2,\dotsc\be_m\}$ of $\bsV$. Using Hadamard's inequality \cite[Cor. 7.8.2]{HJ} we deduce
\[
\lv \det F'(v_0)\rv=\lv\det\lp \pa_{\be_1} F(v_0),\pa_{\be_2}F(v_0),\dotsc, \pa_{\be_m}F(v_0)\rp\rv
\]
\[
\leq \lv \pa_{\be_1}F(v_0)\rv\prod_{k=2}^m\lv \pa_{\be_k}F(v_0)\rv.
\]
Using H\"{o}lder's inequality we deduce 
\[
\bE\lb \lv \det F'(v_0)\rv^2\rv Z=0\; \rb \leq  \prod_{k=1}^{m} \bE\lb\lv \pa_{\be_k}F(v_0)\rv^{2m} \rv\;\lv Z=0\rb\rb^{\frac{1}{m}}.
\]
For $k=2,\dotsc, m$ we have
\[
\Var\lb \pa_{\be_k}F(v_0)\rv\, Z=0\rb\leq \Var\lb \pa_{\be_k}F(v_0)\rb
\]
and 
\[
\lV  \Var\lb \pa_{\be_k}F(v_0)\rb\rV_\op \leq C(m) \Vert \eK_F\Vert_{C^2(B\times B)}.
\]
Using again  Corollary \ref{cor: sup_gauss_int}   we deduce  that for $k=2,\dotsc, m$ we have
\[
\bE\lb\lv \pa_{\be_k}F(v_0)\rv^{2m} \rv\;\lv Z=0 \rb^{\frac{1}{m}}\leq  C(m)\Vert \eK_F\Vert_{C^2(B\times B)}.
\]
Using  (\ref{asymp}), we deduce that
\[
\bE\lb \lv \det F'(v_0)\rv^2\rv\;Z=0\rb \leq C(m) r^2\bE\lb \Vert F\Vert^2_{C^2(B)}\rb\Vert \eK_F\Vert^{m-1}_{C^2(B\times B)}.
\]
Invoking  (\ref{Naz_Sod})  we  conclude that
\[
\bE\lb \Vert F\Vert^2_{C^2(B)}\rb\leq C(m,\vol[B],r) \Vert K\Vert^{3}_{C^6(S\times S)}.
\]
This completes the proof of Lemma \ref{lemma: asymp}.
\end{proof}

Lemma \ref{lemma: asymp} implies
\[
  \bE\lb\vert\det F'(v_0)\det F'(v_1)\vert\;\lv  Z(\uv)=0\rb
 \]
 \[
 \begin{split}
 \leq  \bE\lb \lv \det F'(v_0)\rv^2\rv Z(\uv)=0\rb\rv^{1/2} \bE\lb \lv \det F'(v_1)\rv^2\rv\;Z(\uv)=0\rb\rv^{1/2}
 \end{split}
 \]
 \[
 \leq C(m,\vol[B],r)  \Vert K\Vert^{m-1/2}_{C^6(S\times S)} r(\uv)^{-2}.
 \]
Hence
\begin{equation}\label{blow_diag}
\rho^{(2)}_F(\uv)\leq C(m) \Vert K\Vert^{m-1/2}_{C^6(S\times S)}r(\uv)^{2-m}\sup_{\uv} p_{F(v_0)\oplus \Xi(\uv)}(0).
\end{equation}
Moreover
\[
\sup_{\uv} p_{F(v_0)\oplus \Xi(\uv)}(0)\leq C(m) \Vert K\Vert^{2m}_{C^3(B\times B)}\leq C(m) \Vert K\Vert^{2m}_{C^6(S\times S)}.
\]
This  completes the   proof of Proposition \ref{prop: rad_blow}.
\end{proof}

 We can extract from the above proof a more  precise result. For  any box $B$ in a Euclidean space $\bsV$  we set
 \[
\fq(B):= \int_{B^2_*} r(\uv)^{2-m}dv_0dv_1.
 \]
 Note that $\fq(B)$ is a translation invariant and for any $t>0$, $\fq(tB)=t^{m+2}\fq(B)$. In particular, if $B$ is  the  cube $B_c=[0,c]^m$, then
 \[
 \fq( B_c)= \fq(B_1)c^{m+2}=C(m)\fq(B_1)\vol\lb B_c\rb^{\frac{m+2}{m}}.
 \]

\begin{corollary}\label{cor: bound_var} Let $\eV$ be an open subset of $\bsV$. For each $r>0$    there exists a function 
\[
\mathfrak{F}: (0,\infty)\to(0,\infty)
\]
    with the following property: for any $m_0>0$,  any box  $B\subset \mV$ and any Gaussian field $F:\Omega\times \eV\to\bsU$  such that

\begin{itemize}

\item $\dist(B,\pa\mV)<r$,

\item the covariance kernel $\eK_F$  is $C^6$,  

\item  the restriction of $F$ to $B$ is $2$-ample,

\item   and $ \lV\eK\rV_{C^6(S\times S)} <m_0$
 \end{itemize}
 we have
 \[
\lV\rho^{(2)}_{F}\rV_{L^1(B\times B)}<\mathfrak{F}_r(m_0) \fq(B). 
 \]
 \qed
 \end{corollary}

\begin{remark} One can show that if $F$ is $\as$ $C^3$, then the  function $w_F$ in Proposition \ref{prop: rad_blow} admits  an extension to a continuous  function on the radial blow-up of $B^2$ along the diagonal.\qed
\end{remark}

\end{document}

%% file: Prob_definitions.tex
\newskip\aline \newskip\halfaline
\def\skipaline{\vskip\aline}

\def\qedbox{$\rlap{$\sqcap$}\sqcup$}
\def\qed{\nobreak\hfill\penalty250 \hbox{}\nobreak\hfill\qedbox\skipaline}

\newcommand{\cond}{\,\Vert\,}

\newcommand*\xbar[1]{%
   \hbox{%
     \vbox{%
       \hrule height 0.5pt 
       \kern0.5ex
       \hbox{%
         \kern-0.1em
         \ensuremath{#1}%
         \kern-0.1em
       }%
     }%
   }%
}

\DeclareFontFamily{OT1}{pzc}{}
\DeclareFontShape{OT1}{pzc}{m}{it}{<-> s * [1.10] pzcmi7t}{}
\DeclareMathAlphabet{\mathpzc}{OT1}{pzc}{m}{it}

\newcommand{\one}{{\mathbbm{1}}}

\newcommand\bE{{\mathbb E}}

\newcommand{\bN}{{{\mathbb N}}}

\newcommand{\bP}{{{\mathbb P}}}

\newcommand\bR{{\mathbb R}}

\newcommand{\bT}{{\mathbb T}}

\newcommand{\bV}{{{\mathbb V}}}

\newcommand\bZ{{\mathbb Z}}

\newcommand{\be}{{\boldsymbol{e}}}

\newcommand{\ii}{\boldsymbol{i}}

\newcommand{\bp}{{\boldsymbol{p}}}

\newcommand{\bt}{{\boldsymbol{t}}}
\newcommand{\bu}{{\boldsymbol{u}}}
\newcommand{\bv}{{\boldsymbol{v}}}
\newcommand{\bw}{{\boldsymbol{w}}}
\newcommand{\bx}{{\boldsymbol{x}}}
\newcommand{\by}{{\boldsymbol{y}}}
\newcommand{\bz}{{\boldsymbol{z}}}

\newcommand{\bsI}{\boldsymbol{I}}

\newcommand{\bsK}{{\boldsymbol{K}}}

\newcommand{\bsT}{{\boldsymbol{T}}}
\newcommand{\bsU}{{\boldsymbol{U}}}
\newcommand{\bsV}{{\boldsymbol{V}}}
\newcommand{\bsW}{\boldsymbol{W}}
\newcommand{\bsX}{{\boldsymbol{X}}}
\newcommand{\bsY}{{\boldsymbol{Y}}}

\newcommand{\fC}{\mathfrak{C}}

\newcommand{\fM}{\mathfrak{M}}

\newcommand{\fq}{\mathfrak{q}}

\newcommand{\bGamma}{\boldsymbol{\Gamma}}

\newcommand{\blam}{{\boldsymbol{\lambda}}}

\newcommand{\bnu}{{\boldsymbol{\nu}}}

\newcommand{\ve}{{\varepsilon}}

\newcommand{\vfi}{{\varphi}}

\newcommand{\eC}{\EuScript{C}}

\newcommand{\eE}{\EuScript{E}}

\newcommand{\eI}{{\EuScript{I}}}

\newcommand{\eK}{\EuScript{K}}

\newcommand{\eN}{\EuScript{N}}
\newcommand{\eO}{\EuScript{O}}

\newcommand{\eR}{\EuScript{R}}
\newcommand{\eS}{\EuScript{S}}
\newcommand{\eT}{\EuScript{T}}

\newcommand{\eV}{\EuScript{V}}

\newcommand{\eX}{{\mathfrak{X}}}

\newcommand{\eZ}{\EuScript{Z}}

\newcommand{\mV}{\mathscr{V}}

\DeclareMathOperator{\uu}{\underline{\mathit{u}}}

\DeclareMathOperator{\tr}{{\rm tr}}

\DeclareMathOperator{\supp}{{\rm supp}}

\DeclareMathOperator{\dist}{dist}

\DeclareMathOperator{\vol}{vol}

\DeclareMathOperator{\End}{End}

\DeclareMathOperator{\spa}{span}

\DeclareMathOperator{\ev}{\mathbf{Ev}}

\DeclareMathOperator{\Sym}{\mathbf{Sym}}

\DeclareMathOperator{\Hess}{Hess}

\DeclareMathOperator{\var}{Var}
\DeclareMathOperator{\cov}{Cov}

\DeclareMathOperator{\Var}{Var}
\newcommand{\loc}{{\mathrm{loc}}}

\DeclareMathOperator{\Prob}{Prob}

\DeclareMathOperator{\Meas}{Meas}

\newcommand{\ra}{\rightarrow}

\newcommand{\Ra}{\Rightarrow}

\newcommand{\Llra}{{\Longleftrightarrow}}

\newcommand{\lan}{\langle}
\newcommand{\ran}{\rangle}

\newcommand{\rb}{\,\big]}
\newcommand{\lb}{\big[\,}
\newcommand{\rbr}{\,\big\}}
\newcommand{\lbr}{\big\{\,}

\newcommand{\rp}{\,\big)}
\newcommand{\lp}{\big(\,}
\newcommand{\lv}{\big\vert\,}
\newcommand{\rv}{\,\big\vert}
\newcommand{\Lv}{\Big\vert\,}
\newcommand{\Rv}{\,\Big\vert}

\newcommand{\lV}{\big\Vert\,}
\newcommand{\rV}{\,\big\Vert}

\newcommand{\Rb}{\,\Big]}
\newcommand{\Lb}{\Big[\,}
\newcommand{\Rp}{\,\Big)}
\newcommand{\Lp}{\Big(\,}

\def\inpr{\mathbin{\hbox to 6pt{\vrule height0.4pt width5pt depth0pt \kern-.4pt \vrule height6pt width0.4pt depth0pt\hss}}}

\newcommand{\pa}{\partial}

\newcommand{\hPhi}{{\widehat{\Phi}}}
\newcommand{\hphi}{{\widehat{\Phi}}}

\newcommand{\hh}{\widehat{H}}

\newcommand{\hrho}{\widehat{\rho}}

\usepackage{mathtools}
\def\rbinom#1#2{\ensuremath{\left(\kern-.3em\left(\genfrac{}{}{0pt}{}{#1}{#2}\right)\kern-.3em\right)}}

\newcommand{\uv}{ {\underline{v}} }

\newcommand{\ut}{{\underline{t}}}

\newcommand{\as}{\mathrm{a.s.}}
\newcommand{\iid}{\mathrm{i.i.d.}}

\newcommand{\cpt}{{\mathrm{cpt}}}

\newcommand{\op}{{\mathrm{op}}}

\newcommand{\vtheta}{\vec{\theta}}
\newcommand{\vvfi}{\vec{\varphi}}

\newcommand{\vtau}{{\vec{\tau}}}

\newcommand{\vk}{{\vec{k}}}
\newcommand{\vell}{ {\vec{\ell}} }

\newcommand{\tphi}{{\widetilde{\Phi}}}
\newcommand{\trho}{\widetilde{\rho}}
\newcommand{\tH}{\widetilde{H}}